\documentclass[11pt]{article}
\usepackage[toc,page]{appendix}
\usepackage{stmaryrd}
\usepackage{etoolbox}
\usepackage{pst-node}
\usepackage{tikz-cd} 
\appto\appendix{\addtocontents{toc}{\protect\setcounter{tocdepth}{0}}}
\usepackage{tikz-cd} 
\usepackage{hyperref}
\usepackage{cite}
\usepackage{enumitem}
\usepackage{amsfonts,amsmath, amssymb,latexsym}
\usepackage{mathtools}
\usepackage{multirow,comment}
\def\Z{{\mathbb Z}}

\def\SL{{\rm SL}}

\def\GL{{\rm GL}}

\def\SO{{\rm SO}}

\def\Stab{{\rm Stab}}

\def\Disc{{\rm Disc}}

\def\Aut{{\rm Aut}}

\def\Tr{{\rm Tr}}

\def\dim{{\rm dim}}

\def\Z{{\mathbb Z}}

\def\fz1{{F_{\Z,1}}}

\def\SO{{\rm SO}}

\def\max{{\rm max}}

\usepackage{tikz}
\usepackage{pgfplots}
\usepackage{pgfplots}
\usepgfplotslibrary{groupplots}
\pgfplotsset{compat=1.18}
\pgfplotsset{compat=1.18}
\usepackage{graphicx}
\usepackage{float}
\usepackage{mathrsfs}

\newtheorem{theorem}{Theorem}

\newtheorem{corollary}[theorem]{Corollary}

\newtheorem{lemma}[theorem]{Lemma}

\newtheorem{remark}[theorem]{Remark}

\newtheorem{proposition}[theorem]{Proposition}
\newtheorem{defn}[theorem]{Definition}
\newenvironment{proof}{\noindent {\bf Proof:}}{$\Box$ \vspace{2 ex}}
\newtheorem{question}[theorem]{Question}
\usepackage{xcolor,color,url,lmodern}

\newtheorem{maintheorem}{Theorem}

\newtheorem{conjecture}[theorem]{Conjecture}

\usepackage{xcolor}
\usepackage{environ}
\definecolor{darkpurple}{rgb}{0.5, 0.0, 0.5}
\newif\ifshowcontent
\showcontentfalse

\NewEnviron{highlighted}{
    \ifshowcontent
    \noindent
    {\bfseries\color{darkpurple} \textit{Note:} \BODY}
    \par
  \fi
}
\newenvironment{mainconjecture}[1][]{%
  \par\addvspace{\topsep}\noindent
  \textbf{Conjecture\if\relax\detokenize{#1}\relax\else\ (#1)\fi}%
  \enspace\itshape\ignorespaces
}{%
  \par\addvspace{\topsep}
}
\title{Escape and Non-Escape of Mass for Monogenic Binary Cubic Forms}

\author{Petros Ploumidis}
\date{}

\begin{document}

\maketitle
\begin{abstract}
We study the distribution of monogenic integral binary cubic forms $f(x,y)$ with bounded discriminant and bounded monogenizer inside a fundamental domain for the action of $\GL_2(\mathbb Z)$ on the space of real binary cubic forms. We show that when the monogenizer for $f$ is $[\pm1:0]$ (i.e., $f(\pm1,0)=1$), the corresponding set is fully concentrated in the cusp. We also prove a complementary result, namely, that when the monogenizers $[\pm1:0]$ (with height bounded in terms of the discriminant bound) are excluded, the resulting set has no escape of mass to the cusp. However, we surprisingly show that despite this non-escape of mass, the points in the fundamental domain are not equidistributed.
\end{abstract}

\section{Introduction}

A matrix $\gamma\in\mathrm{GL}_2(\mathbb R)$ acts on the
space of real binary cubic forms by
$(\gamma\cdot f)(x,y)=\det(\gamma)^{-1}f((x,y)\gamma)$.
This is called the twisted action. The space of real binary
cubic forms is prehomogeneous, and its nonzero-discriminant
locus splits into two orbits:
\[
V_{\mathbb R}^{\Disc\neq0}
=
V_{\mathbb R}^{+}\sqcup V_{\mathbb R}^{-}.
\]
For each of these two orbits, we wish to define a fundamental
domain and count all integral monogenic binary cubic forms with
discriminant of size $X$. As we will see later, we choose
a specific binary cubic form $v^\pm$ and obtain an
$n_i$-fold covering $\mathcal Fv^\pm$ of
$\mathrm{GL}_2(\mathbb Z)\backslash V_{\mathbb R}^{(i)}$,
where $\mathcal F$ is the Gauss fundamental domain,
$i=0$ corresponds to positive discriminant, $i=1$ corresponds
to negative discriminant, $n_0=6$ and $n_1=2$. We call a binary cubic form \textit{reduced} if it belongs to
$\mathcal Fv^\pm$.

Two natural questions arise, concerning escape of mass
and equidistribution. The Gauss fundamental domain allows
us to study escape of mass through the diagonal parameter
$a_t=\operatorname{diag}(t^{-1},t)$, loosely speaking, when $t$ is small we are in the main ball and when $t$ is large we are in the cusp.
Additionally, projection to the modular fundamental domain,
equipped with the hyperbolic measure $dx\,dy/y^2$, allows us
to consider potential equidistribution with respect to
this measure. Recall that non-escape of mass
is a necessary condition for equidistribution.
In particular, we are interested in understanding the
distribution of irreducible binary cubic forms in
$\mathcal F_Xv^\pm$ for which there exists an integral
vector $P=(\alpha,\beta)$ satisfying $f(P)=1$ and
$H(P)\ll X^{1/12-\delta}$ for some fixed $\delta>0$.

Fix constants $0<c_-<c_+$. For a vector $P\in\mathbb Z^2$ and $i\in\{0,1\}$, let $N_P^{(i)}(X)$ count  irreducible integral forms in $\mathcal Fv^{(i)}$ satisfying $f(P)=1$ and $c_-X\leq|\Disc(f)|<c_+X$. The index $i=0$ corresponds to positive discriminant and $i=1$ to negative discriminant. The constants $c_-$ and $c_+$ remain fixed as $X\to\infty$.

The cusp is the unbounded end of the modular fundamental domain, where the imaginary part tends to infinity. In the group coordinates introduced in Section~\ref{sec:mass-thue}, we truncate the cusp using the parameter $t$ in $a(t)=\operatorname{diag}(t^{-1},t)$. Let $N_{P,T}^{(i)}(X)$ count the same forms as $N_P^{(i)}(X)$, with the additional condition $t>T$. Full escape of mass means that, as $X$ grows, an asymptotically full proportion of the forms lies beyond every fixed truncation. Non-escape means that, for any prescribed proportion less than the total number of forms, a sufficiently large fixed truncation retains at least that proportion for all sufficiently large $X$. These conditions concern the reduced forms after taking into consideration some finite multiplicities due to the fact that $\mathcal{F}v^{(i)}$ is an $n_i$-fold covering of $\GL_2(\mathbb{Z})\backslash V_{\mathbb{R}}^{(i)}$. More precisely, for a fixed (primitive) point $P\in\mathbb Z^2$, we define 
\[
N_P^{(i)}(X)
=
\#\{f\in\mathcal Fv^{(i)}\cap V_{\mathbb Z}^{(i),\mathrm{irr}}:
f(P)=1,\ c_-X\leq|\Disc(f)|<c_+X\},
\]
and
\[
N_{P,T}^{(i)}(X)
=
\#\{f\in\mathcal F_Tv^{(i)}\cap V_{\mathbb Z}^{(i),\mathrm{irr}}:
f(P)=1,\ c_-X\leq|\Disc(f)|<c_+X\}.
\]
where \[\mathcal F_T
=
\{n(u)a(t)k\lambda\in\mathcal F:t>T\},
\qquad
V_{\mathbb Z}^{(i),\mathrm{irr}}
=
\{f\in V_{\mathbb Z}^{(i)}:f\text{ is irreducible over }\mathbb Q\}.\]
We say that the family of irreducible binary cubic forms in $\mathcal{F}v^{(i)}$ satisfying $f(P)=1$ exhibits non-escape of mass if
\[
\lim_{T\to\infty}\limsup_{X\to\infty}
\frac{N_{P,T}^{(i)}(X)}{N_P^{(i)}(X)}=0,
\]
and (full) escape of mass if
\[
\lim_{T\to\infty}\liminf_{X\to\infty}
\frac{N_{P,T}^{(i)}(X)}{N_P^{(i)}(X)}
=1.
\]
More generally, let $\mathcal P_X\subseteq\mathbb Z_{\mathrm{prim}}^2$.
We say that the family of pairs $(f,P)$ with $P\in\mathcal P_X$
exhibits \emph{averaged non-escape of mass} if
\[
\lim_{T\to\infty}\limsup_{X\to\infty}
\frac{\sum_{P\in\mathcal P_X}N_{P,T}^{(i)}(X)}
     {\sum_{P\in\mathcal P_X}N_P^{(i)}(X)}
=0,
\]
provided the sums are finite and the denominator is positive
for all sufficiently large $X$.
Our first two results give a dichotomy according to the fixed solution.

\begin{maintheorem}[Non-escape for a fixed solution; Theorem~\ref{thm:non-escape-fixed-P}]
For every fixed (primitive) vector $P\neq(\pm1,0)$, and $i\in \{0,1\}$, the family of irreducible binary cubic forms in $\mathcal{F}v^{(i)}$ satisfying $f(P)=1$ exhibits non-escape of mass.
\end{maintheorem}

\begin{maintheorem}[Full escape; Theorem~\ref{thm:escape-e1}]
For $P=(\pm1,0)$, and $i\in \{0,1\}$,
the family of irreducible binary cubic forms in $\mathcal{F}v^{(i)}$ satisfying $f(P)=1$ exhibits (full) escape of mass.
\end{maintheorem}
\noindent
The non-escape estimates are sufficiently uniform to allow the marked solution to vary with $X$.

\begin{maintheorem}[Averaged non-escape; Theorem~\ref{thm:non-escape-averaged}]
Fix $\delta>0$ and $i\in\{0,1\}$, and let $Y=Y(X)$ satisfy $2\leq Y\leq X^{1/12-\delta}$. Define
$\mathcal P(Y)=\{(\alpha,\beta)\in\mathbb Z^2:\gcd(\alpha,\beta)=1,\ Y\leq|\alpha|,|\beta|<2Y\}$.
Then the family of pairs $(f,P)$ with $P\in \mathcal P(Y)$ exhibits averaged non-escape of mass.
\end{maintheorem}

We briefly explain the counting arguments. For $P=(\alpha,\beta)$, the equation $f(P)=1$ defines the affine hyperplane
$\mathcal H_P=\{(a,b,c,d)\in\mathbb R^4:a\alpha^3+b\alpha^2\beta+c\alpha\beta^2+d\beta^3=1\}$.
Its integral points form an affine lattice of rank three. When $\beta\neq0$, the transformed lattice has covolume of order $|\beta|^3t^3$ in the cusp. Applying Schmidt's lattice-point theorem \cite{schmidt1968asymptotic,widmer2012lipschitz} gives
\[
N_{P,T}^{(i)}(X)
\ll_P \frac{X^{3/4}}{T^3}+X^{1/2+\varepsilon},
\qquad
N_P^{(i)}(X)\gg_P X^{3/4}.
\]
For $0<\varepsilon<1/4$, these estimates show that the limiting upper proportion in the cusp is $O_P(T^{-3})$, which tends to zero as $T\to\infty$.

For $P=(\pm1,0)$, the Thue condition fixes the leading coefficient of the form, and the transformed lattice contracts in the cusp. Set $M_{P,T}^{(i)}(X)=N_P^{(i)}(X)-N_{P,T}^{(i)}(X)$, so that $M_{P,T}^{(i)}(X)$ counts forms satisfying the same marking and discriminant conditions with $t\leq T$. A direct lattice-point estimate gives $M_{P,T}^{(i)}(X)\ll X^{3/4}T^3$, while \textit{Bhargava's averaging method} \cite{bhargava2010density} gives $N_P^{(i)}(X)\gg X^{5/6}$, which is a sufficient bound for our purpose (see Remark~\ref{averaging method important}). Consequently, for every fixed $T$,
\[
1-\frac{N_{P,T}^{(i)}(X)}{N_P^{(i)}(X)}
=
\frac{M_{P,T}^{(i)}(X)}{N_P^{(i)}(X)}
\ll T^3X^{-1/12},
\]
which tends to zero as $X\to\infty$. 

For the averaged result, summing the estimates over $\mathcal P(Y)$ gives a main contribution of order $X^{3/4}/Y$, while the accumulated error is bounded by $O_\varepsilon(Y^2X^{1/2+\varepsilon})$. Their ratio is of order $Y^3X^{-1/4+\varepsilon}$, explaining the restriction $Y\leq X^{1/12-\delta}$. The exponent $1/12$ therefore comes from the range in which the counting error is smaller than the main contribution.

We conjecture that non-escape persists when we average over
\textit{all} primitive solutions other than $(\pm1,0)$, \textit{without}
a height restriction.

\begin{mainconjecture}[Non-escape without a height restriction; Conjecture \ref{conj:non-escape-unrestricted}]
\label{conj:non-escape-unrestricted} 
\noindent
Let $i\in\{0,1\}$ and put
\[
\mathcal P=\mathbb Z_{\mathrm{prim}}^2
\setminus\{(1,0),(-1,0)\}.\]
Define
\[
J^{(i)}(X)=\sum_{P\in\mathcal P}N_P^{(i)}(X),
\qquad
J_T^{(i)}(X)=\sum_{P\in\mathcal P}N_{P,T}^{(i)}(X).
\]
Then
\[
\lim_{T\to\infty}\limsup_{X\to\infty}
\frac{J_T^{(i)}(X)}{J^{(i)}(X)}=0.
\]
\end{mainconjecture}
 For the counting arguments in Section~\ref{sec:mass-thue}, we describe reduction through the rank-two lattice $R(f)/\mathbb Z$ associated with the Delone--Faddeev correspondence \cite{Delone2009TheTO}. Section~\ref{sec:df-reduction} defines its parameter $\tau_f$ in the upper half-plane $\mathfrak H=\{z\in\mathbb C:\Im z>0\}$. We call a form reduced when this parameter belongs to the standard modular fundamental domain
$\mathfrak F=\{z\in\mathfrak H:|z|\geq1,\ |\Re z|\leq1/2\}$.
Section~\ref{sec:mass-thue} fixes the group representatives and base forms used in the counts. As before, $\mathcal F$ denotes Gauss's usual fundamental domain
for $\mathrm{GL}_2(\mathbb Z)\backslash\mathrm{GL}_2(\mathbb R)$
in $\mathrm{GL}_2(\mathbb R)$ and $\mathcal Fv^{(i)}$ for the resulting fundamental multiset of binary cubic forms; this distinguishes them from the domain $\mathfrak F$ in the upper half-plane.

Non-escape of mass does \textit{not} necessarily  determine how binary cubic forms are distributed within the fundamental domain.
For each fixed primitive monogenizer $P\neq(\pm1,0)$,
and more generally when the monogenizer has bounded height,
we show that these parameters are not equidistributed
with respect to hyperbolic measure on
$\SL_2(\mathbb Z)\backslash\mathfrak H$.
Whether equidistribution holds without a height restriction remains open. See
Remark~\ref{rem:non-escape-without-equidistribution}
for further details.

An integral Weierstrass model $E:y^2=x^3+Ax+B$ determines the monic binary cubic form $F_E(X,Y)=X^3+AXY^2+BY^3$, which satisfies $F_E(1,0)=1$. For a binary cubic form $f(X,Y)=aX^3+bX^2Y+cXY^2+dY^3$, its discriminant is
$\Disc(f)=b^2c^2-4ac^3-4b^3d-27a^2d^2+18abcd$.
The discriminant of the Weierstrass model therefore satisfies $\Delta(E)=-16(4A^3+27B^2)=16\Disc(F_E)$.

As before, we use the twisted action $(g\cdot f)(X,Y)=\det(g)^{-1}f((X,Y)g)$, with vectors written as rows. If $U\in\SL_2(\mathbb Z)$ carries $F_E$ to a reduced form $f_E=U\cdot F_E$, then $P_E=(1,0)U^{-1}$ satisfies $f_E(P_E)=1$. Thus reduction transports the solution $(1,0)$ to a distinguished primitive integral solution of the reduced form. This leads us to study reduced irreducible integral binary cubic forms with a fixed primitive solution $P$ to the Thue equation $f(P)=1$.

Our numerical experiments concern families satisfying $|A|\asymp H^2$, $|B|\asymp H^3$, and $|\Delta(E)|\asymp H^m$, where $1<m<6$. Here $H$ describes the sizes of the coefficients: the naive height $\max\{|A|^3,|B|^2\}$ is comparable to $H^6$. The notation $\asymp$ means that the quantities are bounded above and below by positive constant multiples of one another, independently of the parameter tending to infinity. Without substantial cancellation between $4A^3$ and $27B^2$, the discriminant also has order $H^6$. The experiments investigate families in which this cancellation produces a smaller discriminant.

The value $m=4$ arises naturally in the construction of these families. For fixed $B$ with $|B|\asymp H^3$, the discriminant restriction gives an admissible interval for $A^3$ of length of order $H^m$. Near $|A|\asymp H^2$, however, consecutive integral cubes are separated by order $H^4$. Thus, when $m<4$, the admissible interval is shorter than the spacing between consecutive cubes. Appendix~\ref{app:thin-families} explains this calculation and a related interpretation in terms of the roots of $x^3+Ax+B$.

In the sampled families, Cremona's reduction algorithm \cite{cremona1999reduction} frequently gives $P_E=(\pm1,0)$ when $m>4$, whereas for $m<4$ the distinguished solution is typically nontrivial but comparatively small. The computations also suggest a connection between the size of this solution and whether the reduced forms move into the cusp. These observations motivate studying families in which $P$ is fixed. 

We also transfer the counting results to normalized integral Weierstrass models. A reduced form $f$ together with a solution $P$ determines, after sending $P$ to $(1,0)$ and normalizing the $X^2Y$-coefficient modulo $3$, a model $E_{f,P}:y^2=x^3+ux^2+Ax+B$, where $u\in\{0,1,2\}$. Section~\ref{Correspondence} makes this correspondence precise and proves the fixed-$P$ bijection. Since $\Delta(E_{f,P})=16\Disc(f)$, the discriminant intervals correspond exactly. The bounds of  Delone, Evertse and Nagell \cite{Evertse1983,delaunay1930darstellung,nagell1928darstellung} of at most twelve solutions to an irreducible cubic Thue equation, controls multiplicities when the marking is forgotten and allows the averaged non-escape result to pass to distinct normalized models.

Finally, Section~\ref{sec:questions} returns to the thin families that motivated the paper. We ask whether the value $m=4$ separates escape from non-escape when $|\Delta(E)|\asymp H^m$, and how the discriminant relates to the size of the distinguished Thue solution.

\subsection{Paper outline}

Section~\ref{sec:df-reduction} recalls the Delone--Faddeev correspondence and describes reduction through the associated rank-two lattice. Section~\ref{sec:mass-thue} defines the counting functions, proves the three main theorems, and states Conjecture~\ref{conj:non-escape-unrestricted} on the unrestricted case. Section~\ref{Correspondence} establishes the correspondence with normalized integral Weierstrass models and transfers the escape and non-escape statements. Section~\ref{sec:questions} presents the questions motivated by the numerical experiments. Appendix~\ref{app:thin-families} contains the construction of the thin families, the root calculations, and the Cremona-reduction computations.

\section*{Acknowledgements}
I am grateful to my advisor, Arul Shankar, for suggesting this problem and for our constructive discussions. I also thank my academic siblings, Fatemehzahra Janbazi, Alexander Slamen and Yun-Chi Tang for their helpful comments on this paper.
\section{The Delone--Faddeev correspondence and geometric reduction}
\label{sec:df-reduction}

The Delone--Faddeev correspondence associates a cubic ring $R(f)$ to an integral binary cubic form $f$. We recall the properties needed for the counting arguments and then describe reduction using the rank-two quotient $R(f)/\mathbb Z$.

A \emph{cubic ring} is a commutative ring with identity that is free of rank $3$ as a $\mathbb Z$-module. We use the twisted action
$(\gamma\cdot f)(x,y)=\det(\gamma)^{-1}f((x,y)\gamma)$
of $\GL_2(\mathbb Z)$ on integral binary cubic forms.

\begin{theorem}[Delone--Faddeev {\cite{Delone2009TheTO}}]
\label{thm:delone-faddeev}
There is a canonical bijection between the $\GL_2(\mathbb Z)$-equivalence classes of integral binary cubic forms and the isomorphism classes of cubic rings.
\end{theorem}

We use the refined formulation of this correspondence due to Gan--Gross--Savin \cite{10.1215/S0012-7094-02-11514-2}. It preserves discriminants: the discriminant of $R(f)$ equals $\Disc(f)$. We will also use the following two properties.

\begin{proposition}[Bhargava--Shankar--Tsimerman {\cite{bhargava2013davenport}}]
\label{prop:domain-irreducible}
The cubic ring $R(f)$ is an integral domain if and only if $f$ is irreducible over $\mathbb Q$.
\end{proposition}

\begin{proposition}[Bhargava--Shankar--Tsimerman {\cite{bhargava2013davenport}}]
\label{prop:automorphism-stabilizer}
For every integral binary cubic form $f$, there is a natural isomorphism
$\Aut(R(f))\simeq\Stab_{\GL_2(\mathbb Z)}(f)$.
\end{proposition}

\subsection{Reduction through the associated lattice}

We describe reduction by associating a point of the upper half-plane to the based rank-two lattice $R(f)/\mathbb Z$. The construction specifies a realization of this lattice in a Euclidean plane. 

Write $\mathfrak H=\{z\in\mathbb C:\Im z>0\}$. The Möbius action of $\SL_2(\mathbb R)$ on $\mathfrak H$ is
$\gamma\cdot z=(az+b)/(cz+d)$
for $\gamma=\left(\begin{smallmatrix}a&b\\c&d\end{smallmatrix}\right)$; see \cite[Chapter~VII]{serre2012course}. We denote the standard modular fundamental domain by
\[
\mathfrak F
=
\left\{
z\in\mathfrak H:
|z|\geq1,\quad
-\frac12\leq\Re z\leq\frac12
\right\}.
\]

\begin{defn}[Shape parameter]
\label{def:shape-parameter}
Let $f\in\mathbb Z[x,y]$ be an irreducible binary cubic form, and let $(R(f),\bar\omega,\bar\theta)$ be the based cubic ring associated to $f$ by the Delone--Faddeev correspondence. Thus,
$R(f)=\langle1,\omega,\theta\rangle_{\mathbb Z}$
and
$R(f)/\mathbb Z=\langle\bar\omega,\bar\theta\rangle_{\mathbb Z}$.
We define a parameter $\tau_f\in\mathfrak H$ as follows.

\smallskip
\noindent\textup{Negative discriminant.}
Suppose that $\Disc(f)<0$. Then
$R(f)\otimes_{\mathbb Z}\mathbb R\simeq\mathbb R\oplus\mathbb C$,
and quotienting by the real line generated by the identity gives
\[
(R(f)/\mathbb Z)\otimes_{\mathbb Z}\mathbb R
\simeq
(\mathbb R\oplus\mathbb C)/\mathbb R(1,1)
\simeq
\mathbb C.
\]
Let $z_\omega,z_\theta\in\mathbb C$ be the images of $\bar\omega,\bar\theta$ under this identification. These vectors are linearly independent over $\mathbb R$, so $z_\theta/z_\omega$ is nonreal. The two choices of complex embedding give conjugate ratios. Choose the embedding for which the ratio has positive imaginary part, and set
$\tau_f=z_\theta/z_\omega\in\mathfrak H$.

\smallskip
\noindent\textup{Positive discriminant.}
Suppose that $\Disc(f)>0$. Then
$R(f)\otimes_{\mathbb Z}\mathbb R\simeq\mathbb R^3$,
and the trace pairing $\langle x,y\rangle=\Tr(xy)$ is positive definite. For $x\in R(f)\otimes_{\mathbb Z}\mathbb R$, write
$x^0=x-\frac{\Tr(x)}{3}\cdot1$
for its trace-zero projection. Set $u=\omega^0$ and $v=\theta^0$. These depend only on $\bar\omega$ and $\bar\theta$, since $(x+n)^0=x^0$ for $n\in\mathbb Z$.

Define
$A_0=\Tr(u^2)$,
$B_0=\Tr(uv)$,
and
$C_0=\Tr(v^2)$.
The positive definiteness of the trace pairing and the linear independence of $u,v$ give $A_0>0$ and $A_0C_0-B_0^2>0$. Set
\[
\tau_f
=
\frac{B_0+i\sqrt{A_0C_0-B_0^2}}{A_0}
\in\mathfrak H.
\]
\end{defn}

The parameter $\tau_f$ depends on the distinguished oriented basis. Its class in $\SL_2(\mathbb Z)\backslash\mathfrak H$ records the shape of the lattice with the realization specified above. The motivation for the previous definition came from the work of Bhargava and Harron in \cite{bhargava2016equidistribution}.

\begin{defn}[Reduced form]
\label{def:lattice-reduction}
\label{def:new-reduction}
An irreducible binary cubic form $f$ is called \emph{reduced} if $\tau_f\in\mathfrak F$.
\end{defn}

With the basis-change convention used for the correspondence, changing the oriented basis by
$\gamma=\left(\begin{smallmatrix}a&b\\c&d\end{smallmatrix}\right)\in\SL_2(\mathbb Z)$
transforms the parameter by
\[
\tau_f\longmapsto
\gamma\cdot\tau_f
=
\frac{a\tau_f+b}{c\tau_f+d}.
\]
Thus changing the oriented integral basis does not change the modular class $[\tau_f]$. Reduction chooses a representative of this class in $\mathfrak F$, with the usual identifications on its boundary.

For the application to Weierstrass models, take
$F_E(X,Y)=X^3+AXY^2+BY^3$
and assume that $F_E$ is irreducible. Choose $U\in\SL_2(\mathbb Z)$ such that $U\cdot\tau_{F_E}\in\mathfrak F$, and set $f_E=U\cdot F_E$. By equivariance,
$\tau_{f_E}=U\cdot\tau_{F_E}\in\mathfrak F$,
so $f_E$ is reduced. The distinguished solution is
$P_E=(1,0)U^{-1}$,
and
$f_E(P_E)=F_E(1,0)=1$.
The matrix $U$ is obtained by composing the translations and inversions used to move $\tau_{F_E}$ into $\mathfrak F$.

\begin{remark}[Relation to covariant reduction]
Cremona's algorithm associates an upper-half-plane point to a binary cubic form using its Hessian or Julia covariant and then reduces that point to $\mathfrak F$ \cite{cremona1999reduction}. The formulation above also uses reduction in the upper half-plane, but describes the parameter through the associated lattice.
\end{remark}

\begin{remark}[The fundamental multiset used for counting]
Let $V_{\mathbb R}$ denote the space of real binary cubic forms, and write
$V_{\mathbb R}^{(0)}=\{f:\Disc(f)>0\}$
and
$V_{\mathbb R}^{(1)}=\{f:\Disc(f)<0\}$.
For each $i\in\{0,1\}$, choose a base form $v^{(i)}\in V_{\mathbb R}^{(i)}$ with $\tau_{v^{(i)}}=\iota$, where $\iota^2=-1$ and $\Im\iota>0$.

We use $\mathcal F$ for the group representatives specified in Section~\ref{sec:mass-thue}, reserving $\mathfrak F$ for the domain in the upper half-plane. Equivariance gives
$\tau_{g\cdot v^{(i)}}=g\cdot\iota$, see \cite[Chapter~VII]{serre2012course} for further details.
The representatives are chosen so that their associated upper-half-plane points lie in $\mathfrak F$. Consequently, forms in $\mathcal Fv^{(i)}$ are reduced in the sense of Definition~\ref{def:lattice-reduction}.

The region $\mathcal Fv^{(i)}$ is used as a fundamental multiset for the action of $\GL_2(\mathbb Z)$ on $V_{\mathbb R}^{(i)}$. The finite stabilizers account for its multiplicities; Proposition~\ref{prop:fundamental-multiplicity} specifies this relation. The subsequent counts use these fixed representatives and multiplicity conventions. In the next section~\ref{sec:mass-thue} we define the counting functions and study escape and non-escape of mass in the regions $\mathcal Fv^{(i)}$.
\end{remark} 

\section{The mass of binary cubic forms of bounded discriminant with a solution to the Thue equation}\label{sec:mass-thue}

In this section, we study the distribution in the fundamental domain $\mathcal{F}v$ of irreducible integral binary cubic forms admitting a prescribed integral solution to the Thue equation $f(x,y)=1$. After recalling the action of $\GL_2$ and fixing Iwasawa coordinates on the fundamental domain, we formulate escape and non-escape of mass in terms of the cusp parameter $t$. The condition $f(\alpha,\beta)=1$ cuts out an affine rank-three lattice in the four-dimensional space of binary cubic forms, so the required estimates reduce to lattice-point counts \cite{siegel2013lectures} in skewed regions of a hyperplane.

The behavior depends sharply on the prescribed solution. If $P=(\alpha,\beta)$ is primitive and $P\neq(\pm1,0)$, then $\beta\neq0$, and the covolume of the transformed lattice grows like $t^3$ in the cusp. This yields an upper bound of order $X^{3/4}T^{-3}$ for the cuspidal contribution and a lower bound of order $X^{3/4}$ for the full count, proving non-escape of mass. If $P=(\pm1,0)$, then the condition $f(P)=1$ is equivalent to fixing the leading coefficient. In this case the transformed lattice has covolume $t^{-3}$, and the mass is driven into the cusp. Bhargava's averaging method \cite{bhargava2010density} gives a lower bound $\gg X^{5/6}$ for the total count, while the contribution of forms  with $t\le T$ is only $O(X^{3/4}T^3)$; this proves escape of mass. Finally, we establish a uniform non-escape statement after averaging over primitive solutions with $Y\leq |\alpha|,|\beta|<2Y$ and $Y\leq X^{1/12-\delta}$.

\subsection{The representation and the fundamental domain}

Let
\[
V_{\mathbb R}
=
\{f(X,Y)=aX^3+bX^2Y+cXY^2+dY^3:a,b,c,d\in\mathbb R\}.
\]
We use the twisted action of $\GL_2(\mathbb R)$ given by
\[
(g\cdot f)(X,Y)
=
\frac{1}{\det(g)}f((X,Y)g).
\]
Under this action,
\[
\Disc(g\cdot f)=\det(g)^2\Disc(f).
\]
Thus the sign of the discriminant is preserved by $\GL_2(\mathbb R)$, while the discriminant itself is preserved by $\GL_2(\mathbb Z)$ \cite{olver1999classical}. We write
\[
V_{\mathbb R}^{(0)}=\{f\in V_{\mathbb R}:\Disc(f)>0\},
\qquad
V_{\mathbb R}^{(1)}=\{f\in V_{\mathbb R}:\Disc(f)<0\},
\]
and
\[
V_{\mathbb Z}^{(i)}=V_{\mathbb Z}\cap V_{\mathbb R}^{(i)},
\qquad i\in\{0,1\}.
\]

We use the Iwasawa coordinates
\[
n(u)=
\begin{pmatrix}
1&0\\
u&1
\end{pmatrix},
\qquad
a(t)=
\begin{pmatrix}
t^{-1}&0\\0&t
\end{pmatrix},
\qquad
k\in K:=\SO_2(\mathbb R),
\qquad
\lambda I_2\in\Lambda,
\]
where $u\in\mathbb R$, $t>0$, and $\lambda>0$.
It is well-known (see [\cite{knapp1996lie}, Theorem~6.46]) that the natural product map
$K_1\times A^+\times N\to\mathrm{GL}_2(\mathbb R)$ is an analytic
diffeomorphism. In fact, for any $g\in\mathrm{GL}_2(\mathbb R)$,
there exist unique $k\in K_1$, $a=a(t)\in A^+$, $n=n(u)\in N$,
and $\lambda\in\Lambda$ such that $g=kan\lambda$; this is the
Iwasawa decomposition of $\mathrm{GL}_2(\mathbb R)$.
We choose a standard set of representatives\footnote{$\mathcal F$ denotes Gauss's usual fundamental domain
for $\mathrm{GL}_2(\mathbb Z)\backslash\mathrm{GL}_2(\mathbb R)$
in $\mathrm{GL}_2(\mathbb R)$.}
\[
\mathcal F
=
\{n(u)a(t)k\lambda:u\in\nu(t),\ t\geq t_0,\ k\in K,\ \lambda>0\},
\qquad
t_0=\frac{\sqrt[4]{3}}{\sqrt2},
\]
for $\GL_2(\mathbb Z)\backslash\GL_2(\mathbb R)$, where $\nu(t)\subset[-\tfrac12,\tfrac12]$ is a union of at most two intervals and $\nu(t)=[-\tfrac12,\tfrac12]$ for $t\geq1$; see \cite[Chapter~7, Theorem~1]{serre2012course}. We normalize the Haar measure as
\[
dg=t^{-2}\,du\,d^{\times}t\,dk\,d^{\times}\lambda,
\]
with $dk$ of total mass $1$.

The scalar parameter $\lambda$ measures the size of the form.
For $g=n(u)a(t)k\lambda$, the discriminant transformation law gives
$|\Disc(g\cdot v)|=\lambda^4|\Disc(v)|$.
Consequently, restricting the discriminant to a fixed interval
$c_-X\leq|\Disc(g\cdot v)|<c_+X$, where $0<c_-<c_+$,
forces $\lambda\asymp X^{1/4}$. The implied constants are uniform
when $v$ ranges over a fixed compact set of forms with nonzero
discriminant.

For $i\in\{0,1\}$, let
\[
n_i=\#\Stab_{\GL_2(\mathbb R)}(v),
\qquad v\in V_{\mathbb R}^{(i)}.
\]
Under our indexing convention,
\[
n_0=6,
\qquad
n_1=2.
\]
These are the orders of $\Aut_{\mathbb R}(\mathbb R^3)$ and $\Aut_{\mathbb R}(\mathbb R\oplus\mathbb C)$ (see \cite{bhargava2013davenport} for more details), respectively.

We recall the asymptotic count of irreducible binary cubic forms without the Thue condition. Let $N(V_{\mathbb Z}^{(i)};X)$ denote the number of irreducible $\GL_2(\mathbb Z)$-orbits in $V_{\mathbb Z}^{(i)}$ having absolute discriminant less than $X$.

\begin{theorem}[Bhargava--Shankar--Tsimerman \cite{bhargava2013davenport}]\label{thm:bst-davenport}
One has
\[
N(V_{\mathbb Z}^{(0)};X)
=
\frac{\pi^2}{72}X+O(X^{5/6}),
\qquad
N(V_{\mathbb Z}^{(1)};X)
=
\frac{\pi^2}{24}X+O(X^{5/6}).
\]
\end{theorem}

Davenport previously obtained the same main terms with error
$O(X^{15/16})$ in \cite{davenport1951principle}.
Theorem~\ref{thm:bst-davenport} counts
$\GL_2(\mathbb Z)$-orbits. For a prescribed point $P$, however,
the condition $f(P)=1$ is \textbf{not} preserved by the action on forms
alone. We therefore count forms satisfying this condition
inside a chosen fundamental multiset $\mathcal Fv^{(i)}$.

The following stabilizer argument explains the finite multiplicity
with which $\mathcal Fv^{(i)}$ represents each orbit.

\begin{proposition}\label{prop:fundamental-multiplicity}
Let $\GL_2(\mathbb R)$ act on a set $X$, let $\mathcal F$ be a set of representatives for $\GL_2(\mathbb Z)\backslash\GL_2(\mathbb R)$, and let $u,x\in X$ with $x\in\GL_2(\mathbb R)\cdot u$. Then there is a natural bijection
\[
\bigsqcup_{y\in\GL_2(\mathbb Z)\cdot x}
\{g\in\mathcal F:g\cdot u=y\}
\longrightarrow
\Stab_{\GL_2(\mathbb Z)}(x)\backslash
\Stab_{\GL_2(\mathbb R)}(x).
\]
\end{proposition}

\begin{proof}
Choose $h \in \GL_2(\mathbb{R})$ such that $hu = x$. For each $y \in \GL_2(\mathbb{Z}) \cdot x$, fix $\gamma_y \in \GL_2(\mathbb{Z})$ with $\gamma_y x = y$. We define a map by sending $g \in \mathcal{F}$ with $gu = y$ to the coset $\operatorname{Stab}_{\GL_2(\mathbb{Z})}(x)\,\gamma_y^{-1} g h^{-1}$. If $gu = y$, then $(\gamma_y^{-1} g h^{-1})x = \gamma_y^{-1} g u = \gamma_y^{-1} y = x$, so the map indeed takes values in $\operatorname{Stab}_{\GL_2(\mathbb{R})}(x)$. It is straightforward to verify that the map is well defined, since it does not depend on the choice of $\gamma_y$.

To prove injectivity, suppose two elements $g_1,g_2 \in \mathcal{F}$, corresponding to $y,z \in \GL_2(\mathbb{Z}) \cdot x$, have the same image. Then there exists $\gamma \in \operatorname{Stab}_{\GL_2(\mathbb{Z})}(x)$ such that $\gamma_y^{-1} g_1 h^{-1} = \gamma\,\gamma_z^{-1} g_2 h^{-1}$, hence $\gamma_y^{-1} g_1 = \gamma\,\gamma_z^{-1} g_2$. It follows that $g_1 g_2^{-1} \in \GL_2(\mathbb{Z})$, and since $\mathcal{F}$ is a set of representatives for $\GL_2(\mathbb{Z}) \backslash \GL_2(\mathbb{R})$, we must have $g_1 = g_2$, and therefore $y=z$.

For surjectivity, let $\operatorname{Stab}_{\GL_2(\mathbb{Z})}(x)\,k$ be any coset with $k \in \operatorname{Stab}_{\GL_2(\mathbb{R})}(x)$. Since $\mathcal{F}$ is a set of representatives, we may write $k = \gamma^{-1} g h^{-1}$ for some $\gamma \in \GL_2(\mathbb{Z})$ and $g \in \mathcal{F}$. Then setting $y = gu$, we obtain $y = \gamma x$, so $y \in \GL_2(\mathbb{Z}) \cdot x$, and one checks that $\gamma_y^{-1} g h^{-1}$ lies in the coset $\operatorname{Stab}_{\GL_2(\mathbb{Z})}(x)\,k$. This shows surjectivity.

Hence the map is a bijection.
\end{proof}

For each $i$, fix a base form $v^{(i)}\in V_{\mathbb R}^{(i)}$
whose shape parameter is $\iota\in\mathfrak H$, where
$\iota^2=-1$. We use the fundamental multiset
$\mathcal Fv^{(i)}$ throughout this section, with the
multiplicities described in
Proposition~\ref{prop:fundamental-multiplicity}. Since a binary cubic form has odd degree, replacing $P$ by $-P$ changes the sign of $f(P)$; it is therefore enough to consider $f(P)=1$.

For $T>t_0$, define the cusp and its complement by
\[
\mathcal F_T
=
\{n(u)a(t)k\lambda\in\mathcal F:t>T\},
\qquad
\mathcal F_{\leq T}
=
\mathcal F\setminus\mathcal F_T.
\]
Let
\[
V_{\mathbb Z}^{(i),\mathrm{irr}}
=
\{f\in V_{\mathbb Z}^{(i)}:f\text{ is irreducible over }\mathbb Q\}.
\]
Fix constants $0<c_-<c_+$. Throughout the counting arguments
below, we use the discriminant interval
$c_-X\leq|\Disc(f)|<c_+X$.
The implied constants may depend on $c_-$ and $c_+$.
For a fixed primitive point $P\in\mathbb Z^2$, set
\[
N_P^{(i)}(X)
=
\#\{f\in\mathcal Fv^{(i)}\cap V_{\mathbb Z}^{(i),\mathrm{irr}}:
f(P)=1,\ c_-X\leq|\Disc(f)|<c_+X\},
\]
and
\[
N_{P,T}^{(i)}(X)
=
\#\{f\in\mathcal F_Tv^{(i)}\cap V_{\mathbb Z}^{(i),\mathrm{irr}}:
f(P)=1,\ c_-X\leq|\Disc(f)|<c_+X\}.
\]
We say that the family of irreducible binary cubic forms in $\mathcal{F}v^{(i)}$ satisfying $f(P)=1$ exhibits non-escape of mass if
\[
\lim_{T\to\infty}\limsup_{X\to\infty}
\frac{N_{P,T}^{(i)}(X)}{N_P^{(i)}(X)}=0,
\]
and (full) escape of mass if
\[
\lim_{T\to\infty}\liminf_{X\to\infty}
\frac{N_{P,T}^{(i)}(X)}{N_P^{(i)}(X)}
=1.
\]

More generally, let $\mathcal P_X\subseteq\mathbb Z_{\mathrm{prim}}^2$.
We say that the family of pairs $(f,P)$ with $P\in\mathcal P_X$
exhibits \emph{averaged non-escape of mass} if
\[
\lim_{T\to\infty}\limsup_{X\to\infty}
\frac{\sum_{P\in\mathcal P_X}N_{P,T}^{(i)}(X)}
     {\sum_{P\in\mathcal P_X}N_P^{(i)}(X)}
=0,
\]
provided the sums are finite and the denominator is positive
for all sufficiently large $X$.

For $P=(\alpha,\beta)$, define
\[
H_{\alpha,\beta}
=
\{(a,b,c,d)\in\mathbb R^4:
a\alpha^3+b\alpha^2\beta+c\alpha\beta^2+d\beta^3=1\},
\]
\[
\mathcal L_{\alpha,\beta}
=
H_{\alpha,\beta}\cap\mathbb Z^4,
\]
and
\[
\Lambda_{\alpha,\beta}
=
\{(a,b,c,d)\in\mathbb Z^4:
a\alpha^3+b\alpha^2\beta+c\alpha\beta^2+d\beta^3=0\}.
\]
Thus the condition $f(P)=1$ turns the counting problem into a lattice-point problem on the affine slice $\mathcal L_{\alpha,\beta}$.

\subsection{Lattice-point estimates}

Our principal counting tool is Schmidt's theorem. A bounded
measurable set is of narrow class $s$\cite{widmer2012lipschitz} if its intersection with
every line has at most $s$ connected components, and the same
property holds for every orthogonal projection of the set
onto a linear subspace.

\begin{theorem}[Schmidt,\cite{widmer2012lipschitz}]\label{thm:schmidt}
Let $\Lambda\subset\mathbb R^n$ be a lattice with successive minima $\lambda_1,\ldots,\lambda_n$. Let $S\subseteq B_0(R)$ be of narrow class $s$. Then
\[
\left|
\#(S\cap\Lambda)-\frac{\operatorname{Vol}(S)}{\det\Lambda}
\right|
\leq
c_1(n,s)
\max_{0\leq j<n}
\frac{R^j}{\lambda_1\cdots\lambda_j},
\]
where the term for $j=0$ is $1$, and $c_1(n,s)$ depends only
on the dimension $n$ and the narrow-class bound $s$.
\end{theorem}

We pass from the affine lattice to a homogeneous lattice
using the following elementary observation.

\begin{lemma}[Affine lattices]\label{lem:affine-translate}
Let $q=(x_1,\ldots,x_n)\in\mathbb Z^n$ be primitive, and define
\[
\widetilde\Lambda_q=\{m\in\mathbb Z^n:m\cdot q=1\},
\qquad
\Lambda_q=\{m\in\mathbb Z^n:m\cdot q=0\}.
\]
Then there exists $q_0\in\widetilde\Lambda_q$ such that
\[
\widetilde\Lambda_q=\Lambda_q+q_0.
\]
Consequently, for every bounded set $\mathcal B\subset\mathbb R^n$,
\[
\#(\widetilde\Lambda_q\cap\mathcal B)
=
\#(\Lambda_q\cap(\mathcal B-q_0)).
\]
\end{lemma}

\begin{proof}
Choose $q_0\in\mathbb Z^n$ with $q_0\cdot q=1$ by B\'ezout's identity. Then $m\cdot q=1$ if and only if $(m-q_0)\cdot q=0$.
\end{proof}

\begin{remark}\label{constant affects}
    Translation preserves the volume of the region and the
covolume of the lattice. However, it can increase the radius
of a ball centered at the origin containing the region,
and hence it can affect \textbf{only} the constant of the error estimate in Schmidt's
theorem. 
\end{remark}

We next estimate the reducible forms satisfying $f(P)=1$,
so that they can be removed from the lattice-point counts. To do that, we first record the coefficient bounds that we will use throughout this section.

\begin{lemma}\label{lem:iwasawa-coefficients}
Let $v$ range over a fixed compact subset of $V_{\mathbb R}^{(i)}$, and write
\[
f=n(u)a(t)k\lambda\cdot v
=ax^3+bx^2y+cxy^2+dy^3.
\]
Uniformly for $u$ in a bounded interval and $k\in K$, one has
\[
|a|\ll\frac{\lambda}{t^3},
\qquad
|b|\ll\frac{\lambda}{t},
\qquad
|c|\ll\lambda t,
\qquad
|d|\ll\lambda t^3.
\]
\end{lemma}

\begin{proof}
Write
\[
k\cdot v=a_kx^3+b_kx^2y+c_kxy^2+d_ky^3.
\]
The coefficients $a_k,b_k,c_k,d_k$ are uniformly bounded. Applying $a(t)$ and $\lambda$, and then substituting $x+uy$ for $x$, gives
\begin{align*}
a&=\lambda\frac{a_k}{t^3},\\
b&=\lambda\left(\frac{3ua_k}{t^3}+\frac{b_k}{t}\right),\\
c&=\lambda\left(\frac{3u^2a_k}{t^3}+\frac{2ub_k}{t}+c_kt\right),\\
d&=\lambda\left(\frac{u^3a_k}{t^3}+\frac{u^2b_k}{t}+uc_kt+d_kt^3\right).
\end{align*}
The stated bounds follow.
\end{proof}

\begin{lemma}[Reducible forms]\label{lem:reducible-fixed-P}
Let $B\subset V_{\mathbb R}^{(i)}$ be compact and consist of forms of nonzero discriminant. Uniformly for $u^{(i)}\in B$ and primitive $P=(\alpha,\beta)\in\mathbb Z^2$, the number of reducible forms
\[
f(x,y)=ax^3+bx^2y+cxy^2+dy^3
\in
\mathcal Fu^{(i)}\cap V_{\mathbb Z}
\]
with $c_-X\leq|\Disc(f)|<c_+X$, $f(P)=1$, and $a\neq0$ is
\[
O(X^{1/2+\varepsilon}).
\]
The implied constant depends only on $B$ and $\varepsilon$.
\end{lemma}

\begin{proof}
Write $f=g\cdot u^{(i)}$ with $g=n(u)a(t)k\lambda\in\mathcal F$. Since $c_-X\leq|\Disc(f)|<c_+X$, we have $\lambda\asymp X^{1/4}$. The standard coefficient estimates \ref{lem:iwasawa-coefficients} give
\[
|a|\ll X^{1/4},
\qquad
|ab|,|ac|,|ad|\ll X^{1/2},
\qquad
|abc|,|abd|\ll X^{3/4}.
\]
If $d=0$, then $f(P)=1$ determines $c$ once $a$ and $b$ are fixed, and hence the number of possibilities is 
\[
\sum_{0<|a|\ll X^{1/4}}
\sum_{|ab|\ll X^{1/2}}1
\ll X^{1/2+\varepsilon}.
\]
If $d\neq0$ and $f$ is reducible, write
\[
f(x,y)=(rx+sy)(ux^2+vxy+wy^2).
\]
Since $f(\alpha,\beta)=1$, both factors take the value $\pm1$ at $P$. Once $(a,d,r)$ is fixed, the relations
\[
a=ru,
\qquad
d=sw,
\qquad
r\alpha+s\beta=\pm1
\]
determine $s,u,w$, and the remaining unit equation determines $v$. Since $r\mid a$, the number of possibilities is
\[
\sum_{0<|a|\ll X^{1/4}}
\sum_{|ad|\ll X^{1/2}}
\sum_{r\mid a}1
\ll X^{1/2+\varepsilon}.
\]
The cases $\alpha=0$ or $\beta=0$ are obtained similarly. This proves the lemma.
\end{proof}

\subsection{A fixed solution $P\neq(\pm1,0)$}

\begin{theorem}\label{thm:non-escape-fixed-P}
Let $P=(\alpha,\beta)\in\mathbb Z^2$ be primitive with $P\neq(\pm1,0)$, and let $i\in\{0,1\}$. Then the family of reduced irreducible forms in $V_{\mathbb Z}^{(i)}$ satisfying $f(P)=1$ exhibits non-escape of mass. More precisely,
\[
\lim_{T\to\infty}\limsup_{X\to\infty}
\frac{N_{P,T}^{(i)}(X)}{N_P^{(i)}(X)}=0.
\]
\end{theorem}

Let $P=(\alpha,\beta)\in\mathbb Z^2$ be primitive and assume $P\neq(\pm1,0)$. Then $\beta\neq0$. The proof rests on a cuspidal upper bound and a lower bound on the total number of forms. We first determine the geometry of the lattice $\Lambda_{\alpha,\beta}$. Facts about lattices that are stated without proof or reference may be found in \cite{schmidt1968asymptotic} or
any other basic text on the geometry of numbers.
Recall that 
\[
\Lambda_{\alpha,\beta}
=
\{(a,b,c,d)\in\mathbb Z^4:
a\alpha^3+b\alpha^2\beta+c\alpha\beta^2+d\beta^3=0\}.
\]
\begin{lemma}\label{lem:lattice-fixed-P}
The vectors
\[
v_1=(\beta,-\alpha,0,0),
\qquad
v_2=(0,\beta,-\alpha,0),
\qquad
v_3=(0,0,\beta,-\alpha)
\]
form a basis of $\Lambda_{\alpha,\beta}$. In particular,
\[
\dim E_{\Lambda_{\alpha,\beta}}=3
\]
and
\[
d(\Lambda_{\alpha,\beta})
=
\left\|(\alpha^3,\alpha^2\beta,\alpha\beta^2,\beta^3)\right\|,
\]
where $d(\Lambda)$ denotes the covolume of $\Lambda$ in its
real linear span $E_\Lambda=\operatorname{span}_{\mathbb R}(\Lambda)$,
with respect to the induced Euclidean measure.
\end{lemma}

\begin{proof}
A form $F(X,Y)=aX^3+bX^2Y+cXY^2+dY^3$ lies in $\Lambda_{\alpha,\beta}$ precisely when $F(\alpha,\beta)=0$. Since $\gcd(\alpha,\beta)=1$, Gauss's lemma implies that $\beta X-\alpha Y$ divides $F$ in $\mathbb Z[X,Y]$. Thus
\[
F(X,Y)
=
(\beta X-\alpha Y)(xX^2+yXY+zY^2),
\]
and comparison of coefficients gives
\[
(a,b,c,d)=xv_1+yv_2+zv_3.
\]
The reverse inclusion and linear independence are immediate. Moreover,\footnote{A $t$-dimensional integral lattice $\Lambda\subseteq\mathbb Z^m$
is called primitive if $\Lambda=E_\Lambda\cap\mathbb Z^m$.
Its orthogonal lattice is
$\Lambda^\perp=E_\Lambda^\perp\cap\mathbb Z^m$,
which is a primitive integral lattice of dimension $m-t$. Recall that if $\Lambda\subset\mathbb{Z}^m$ is primitive then $d(\Lambda)=d(\Lambda^\perp)$.}
\[
\Lambda_{\alpha,\beta}
=
(\alpha^3,\alpha^2\beta,\alpha\beta^2,\beta^3)^\perp.
\]
The normal vector is primitive, and the equality of covolumes of a primitive lattice and its perpendicular lattice (this useful proof appears in \cite{siegel2013lectures}) yields the final formula.
\end{proof}

For $t\geq1$, let
\[
\eta_t
=
\begin{pmatrix}
t&0\\0&t^{-1}
\end{pmatrix}.
\]
Its action on coefficient space is
\[
\eta_t(a,b,c,d)
=
(t^3a,tb,t^{-1}c,t^{-3}d).
\]
Set
\[
\Lambda_{\alpha,\beta}^t
=
\eta_t\Lambda_{\alpha,\beta}.
\]
The following lemma will be useful when we apply Schmidt's theorem~\ref{thm:schmidt} later.

\begin{lemma}\label{lem:transformed-lattice-fixed-P}
The lattice $\Lambda_{\alpha,\beta}^t$ has basis
\[
v_1^t=(\beta t^3,-\alpha t,0,0),
\qquad
v_2^t=(0,\beta t,-\alpha t^{-1},0),
\]
\[
v_3^t=(0,0,\beta t^{-1},-\alpha t^{-3}).
\]
Moreover,
\[
E_{\Lambda_{\alpha,\beta}^t}
=
(\alpha^3t^{-3},\alpha^2\beta t^{-1},
\alpha\beta^2t,\beta^3t^3)^\perp
\]
and
\[
|\beta|^3t^3
\leq
d(\Lambda_{\alpha,\beta}^t)
\leq
(|\alpha|+|\beta|)^3t^3.
\]
The product $\|v_1^t\|\|v_2^t\|\|v_3^t\|$ is comparable to $d(\Lambda_{\alpha,\beta}^t)$, with constants depending only on $P$.
\end{lemma}

\begin{proof}
The displayed basis follows by applying $\eta_t$ to the basis in Lemma~\ref{lem:lattice-fixed-P}. The transformed normal vector is
\[
(\alpha^3t^{-3},\alpha^2\beta t^{-1},
\alpha\beta^2t,\beta^3t^3),
\]
whose norm equals the covolume. The stated upper and lower bounds follow directly, and the comparison with the product of the basis lengths is immediate from the explicit formulas.
\end{proof}

We use $\Lambda_{\alpha,\beta}^t$ instead of $\Lambda_{\alpha,\beta}$
because it is more convenient to work with a region contained in a
ball of radius $O(\lambda)$ than with a region whose
dimensions are of orders $\lambda/t^3$, $\lambda/t$, $\lambda t$,
and $\lambda t^3$ (recall Lemma~\ref{lem:iwasawa-coefficients}).
We now isolate the two estimates that imply non-escape.

Recall that for a fixed primitive point $P\in\mathbb Z^2$, 
\[
N_P^{(i)}(X)
=
\#\{f\in\mathcal Fv^{(i)}\cap V_{\mathbb Z}^{(i),\mathrm{irr}}:
f(P)=1,\ c_-X\leq|\Disc(f)|<c_+X\},
\]
and
\[
N_{P,T}^{(i)}(X)
=
\#\{f\in\mathcal F_Tv^{(i)}\cap V_{\mathbb Z}^{(i),\mathrm{irr}}:
f(P)=1,\ c_-X\leq|\Disc(f)|<c_+X\}.
\]
\begin{proposition}[Cuspidal upper bound]\label{prop:cusp-upper-fixed-P}
For every $\varepsilon>0$ and every fixed $T>t_0$,
\[
N_{P,T}^{(i)}(X)
\ll
\frac{X^{3/4}}{T^3|\beta|^3}
+
\frac{X^{1/2}}{|\beta|^2}
+
X^{1/2+\varepsilon},
\]
uniformly for primitive $P=(\alpha,\beta)$ with $\beta\neq0$.
\end{proposition}

\begin{proof}
By Lemma~\ref{lem:reducible-fixed-P}, it is enough to count \textit{all}
integral binary cubic forms with nonzero leading coefficient, at the cost of
$O(X^{1/2+\varepsilon})$. Write
\[
f(X,Y)=a_fX^3+b_fX^2Y+c_fXY^2+d_fY^3.
\]
Recall that $\lambda\asymp X^{1/4}$. By
Lemma~\ref{lem:iwasawa-coefficients}, the condition $a_f\neq0$
implies
\[
1\ll \frac{\lambda}{t^3},
\qquad\text{and hence}\qquad
t\ll \lambda^{1/3}.
\]

We now divide the range
\[
T\leq t\ll \lambda^{1/3}
\]
into dyadic intervals
\[
S\leq t<2S,
\qquad
S=2^jT.
\]
For such an $S$, let $\mathcal A_S(P;X)$ denote the contribution to
$N_{P,T}^{(i)}(X)$ from those representatives
\[
f=n(u)a(t)k\lambda\cdot v^{(i)}
\]
satisfying
\[
S\leq t<2S,
\qquad
c_-X\leq|\Disc(f)|<c_+X,
\qquad
f(P)=1,
\qquad
a_f\neq0.
\]
Equivalently,
\[
\mathcal A_S(P;X)
=
\left\{
f\in V_{\mathbb Z}^{(i)}:
\begin{array}{l}
f=n(u)a(t)k\lambda\cdot v^{(i)}
\text{ for some } n(u)a(t)k\lambda\in\mathcal F,\\
S\leq t<2S,\quad c_-X\leq|\Disc(f)|<c_+X,\\
f(P)=1,\quad a_f\neq0
\end{array}
\right\},
\]
where the set is interpreted with the multiplicity coming from the
fundamental multiset.

For every form in this dyadic range,
Lemma~\ref{lem:iwasawa-coefficients} gives the coefficient bounds
\[
|a_f|\ll \frac{\lambda}{S^3},
\qquad
|b_f|\ll \frac{\lambda}{S},
\qquad
|c_f|\ll \lambda S,
\qquad
|d_f|\ll \lambda S^3.
\]
Thus the real region containing the coefficient vectors
$(a_f,b_f,c_f,d_f)$ has side lengths
\[
\ll
\frac{\lambda}{S^3},
\quad
\frac{\lambda}{S},
\quad
\lambda S,
\quad
\lambda S^3.
\]
Recall that, by Lemma~\ref{constant affects}, we can translate $\mathcal{L}_{\alpha,\beta}$ to $\Lambda_{\alpha,\beta}$ and that  
\[
\eta_S(a,b,c,d)
=
(S^3a,Sb,S^{-1}c,S^{-3}d).
\]
Applying the invertible linear map $\eta_S$ to both the above
real region and the lattice $\Lambda_{\alpha,\beta}$ preserves
the lattice-point count.\footnote{Indeed, $\eta_S$ induces a
bijection between the lattice points in the original region
and the points of the transformed lattice in the transformed
region.} The transformed region has side lengths
\[
\ll
\lambda,
\quad
\lambda,
\quad
\lambda,
\quad
\lambda.
\]
So, it is contained in a ball of radius $O(\lambda)$.
Moreover, by Lemma~\ref{lem:transformed-lattice-fixed-P}, the
lattice
\[
\eta_S\{f\in V_{\mathbb Z}:f(P)=0\}
\]
has covolume
\[
\gg |\beta|^3S^3.
\]
After computing all the orthogonal projections  using Lemma~\ref{lem:lattice-fixed-P} for the lattice  $\eta_t\{f\in V_{\mathbb Z}:f(P)=0\}$ and applying Schmidt's theorem~\ref{thm:schmidt}, we get\footnote{We use three-dimensional Hausdorff measure to measure
volume in the three-dimensional affine subspace of $\mathbb R^4$;
see \cite{widmer2012lipschitz} for more details. Similarly, the volumes of
the projections appearing in the error term are measured using
Hausdorff measure of the corresponding dimension. For background
on Hausdorff measure, see \cite{evans2025measure}.}
\[
\#\mathcal A_S(P;X)
\ll
\frac{\lambda^3}{|\beta|^3S^3}
+
\frac{\lambda^2}{|\beta|^2}
+
1.
\]
Summing over $S_j=2^jT$, for $0\leq j\leq J$, where
$2^JT\ll\lambda^{1/3}$ and hence $J\ll\log X$, gives
\[
N_{P,T}^{(i)}(X)
\ll
\frac{\lambda^3}{|\beta|^3T^3}
\sum_{j=0}^{J}2^{-3j}
+
\frac{\lambda^2}{|\beta|^2}(J+1)
+
(J+1)
+
X^{1/2+\varepsilon}.
\]
The geometric sum satisfies
\[
\sum_{j=0}^{J}2^{-3j}
=
\frac{1-2^{-3(J+1)}}{1-2^{-3}}
\leq \frac{8}{7}.
\]
Since $J+1\ll\log X$, we obtain
\[
N_{P,T}^{(i)}(X)
\ll
\frac{\lambda^3}{T^3|\beta|^3}
+
\frac{\lambda^2\log X}{|\beta|^2}
+
\log X
+
X^{1/2+\varepsilon}.
\]
Using $\lambda\asymp X^{1/4}$ and that
the logarithmic terms are absorbed into
$X^{1/2+\varepsilon}$. This gives the stated bound
\[
N_{P,T}^{(i)}(X)
\ll
\frac{X^{3/4}}{T^3|\beta|^3}
+
\frac{X^{1/2}}{|\beta|^2}
+
X^{1/2+\varepsilon}.
\]
\end{proof}

\begin{proposition} [Total lower bound]\label{prop:compact-lower-fixed-P}
For every $\varepsilon>0$,
\[
N_P^{(i)}(X)
\geq
c\frac{X^{3/4}}{(|\alpha|+|\beta|)^3}
-
C\frac{X^{1/2}}{|\beta|^2}
-
C_\varepsilon X^{1/2+\varepsilon},
\]
where the constants are uniform for primitive $P=(\alpha,\beta)$ with $\beta\neq0$. In particular, for fixed $P$,
\[
N_P^{(i)}(X)\gg_P X^{3/4}.
\]
\end{proposition}

\begin{proof}
As in Proposition~\ref{prop:cusp-upper-fixed-P}, we first apply Lemma~\ref{lem:reducible-fixed-P} and we count \textit{all} integral binary cubic forms with nonzero leading coefficient, at the cost of $O(X^{1/2+\varepsilon})$. Additionally, Lemma~\ref{lem:iwasawa-coefficients} forces $t\ll\lambda^{1/3}$. Continuity of the discriminant allows us to choose a compact subregion $\mathcal C_X\subset\mathcal Fv^{(i)}$ with $t$ restricted to a fixed compact interval $[t_1,t_2]\subset[t_0,c\lambda^{1/3})$, small enough that every form in $\mathcal C_X$ belongs to $V_{\mathbb R}^{(i)}$ and satisfies
$c_-X\leq|\Disc(f)|<c_+X$. Since $t,u,k$ remain in a fixed compact range, the coefficient coordinates of forms in $\mathcal C_X$ have side lengths comparable to
\[
\frac{\lambda}{t^3},  \qquad \frac{\lambda}{t},\qquad\lambda t,\qquad\lambda t^3.
\]
After translating $\mathcal L_{\alpha,\beta}$ to $\Lambda_{\alpha,\beta}$ due to Remark~\ref{constant affects} and applying $\eta_{t}$ to both $\mathcal C_X$ and $\Lambda_{\alpha,\beta}$ we get that  the coefficient coordinates of forms in $\eta_t \mathcal C_X$ have side lengths comparable to
\[
\lambda,  \qquad \lambda,\qquad\lambda ,\qquad\lambda .
\]
Thus, the region $\eta_t \mathcal C_X$ contains a ball of radius $O(\lambda)$. Additionally, by Lemma~\ref{lem:transformed-lattice-fixed-P}
the
lattice
\[
\eta_t\{f\in V_{\mathbb Z}:f(P)=0\}
\]
has covolume
\[
\ll
(|\alpha|+|\beta|)^3t^3.
\]
After computing all the orthogonal projections  using Lemma~\ref{lem:lattice-fixed-P} for the lattice  $\eta_t\{f\in V_{\mathbb Z}:f(P)=0\}$ and applying Schmidt's theorem~\ref{thm:schmidt}, we obtain\footnote{We use three-dimensional Hausdorff measure to measure
volume in the three-dimensional affine subspace of $\mathbb R^4$;
see \cite{widmer2012lipschitz} for more details. Similarly, the volumes of
the projections appearing in the error term are measured using
Hausdorff measure of the corresponding dimension. For background
on Hausdorff measure, see \cite{evans2025measure}.}
\[
N_P^{(i)}(X)
\geq
c\frac{\lambda^3}{(|\alpha|+|\beta|)^3}
-
C\frac{\lambda^2}{|\beta|^2}
-
C'_{\epsilon}X^{1/2+\epsilon}.
\]
where $c$, $C$ and $C'$ are some positive constants.
Recall that $\lambda\asymp X^{1/4}$ and we have proved the proposition.
\end{proof}

We can now give the proof of Theorem~\ref{thm:non-escape-fixed-P}.

\begin{proof}
For fixed $P$, Propositions~\ref{prop:cusp-upper-fixed-P} and~\ref{prop:compact-lower-fixed-P} imply
\[
N_{P,T}^{(i)}(X)
\ll_P
\frac{X^{3/4}}{T^3}+X^{1/2+\varepsilon},
\qquad
N_P^{(i)}(X)\gg_P X^{3/4}.
\]
Taking $0<\varepsilon<1/4$ gives
\[
\limsup_{X\to\infty}
\frac{N_{P,T}^{(i)}(X)}{N_P^{(i)}(X)}
\ll_P T^{-3}.
\]
Letting $T\to\infty$ proves the theorem.
\end{proof}

\subsection{The solutions $P=(\pm1,0)$}

\begin{theorem}\label{thm:escape-e1}
Let $P=(\alpha,0)$ with $\alpha\in\{\pm1\}$, and let $i\in\{0,1\}$. Then the family of reduced irreducible forms in $V_{\mathbb Z}^{(i)}$ satisfying $f(P)=1$ exhibits escape of mass. More precisely,
\[
\lim_{T\to\infty}\liminf_{X\to\infty}
\frac{N_{P,T}^{(i)}(X)}{N_P^{(i)}(X)}
=1.
\]
\end{theorem}
It suffices to show that 
\[
\lim_{T\to\infty}\limsup_{X\to\infty}
\left|
\frac{N_{P,T}^{(i)}(X)}{N_P^{(i)}(X)}-1
\right|=\lim_{T\to\infty}\limsup_{X\to\infty}
\left|
\frac{N^{(i)}_{P}(X)-N_{P,T}^{(i)}(X)}{N_P^{(i)}(X)}
\right|=0.
\]
Let $P=(\alpha,0)$ with $\alpha\in\{\pm1\}$. If
\[
f(X,Y)=aX^3+bX^2Y+cXY^2+dY^3,
\]
then $f(\alpha,0)=a\alpha$. Hence the condition $f(P)=1$ is equivalent to $a=\alpha$. Thus the exceptional families are exactly the slices of coefficient space with fixed leading coefficient $a=\alpha$.

\begin{remark}\label{averaging method important}
    For these families, a lower bound of order $X^{3/4}$ would
\textbf{not} suffice to prove escape of mass: our upper bound for the
count with $t\leq T$ is also of order $X^{3/4}$ for fixed $T$.
Bhargava's averaging method \cite{bhargava2010density} gives us  the stronger lower bound
$\gg X^{5/6}$ for the total count.
\end{remark}

We write $\mathcal F_T$ for the part of the fundamental domain with cusp parameter $t>T$, and $\mathcal F_{\leq T}$ for its complement. Define
\[
M_{(\alpha,0),T}^{(i)}(X)
=
\#\{f\in\mathcal F_{\leq T}v^{(i)}\cap V_{\mathbb Z}^{(i),\mathrm{irr}}:
a(f)=\alpha,\ c_-X\leq|\Disc(f)|<c_+X\}.
\]
Then
\[
N_{(\alpha,0)}^{(i)}(X)
=
N_{(\alpha,0),T}^{(i)}(X)
+
M_{(\alpha,0),T}^{(i)}(X).
\]

Recall that the affine lattice corresponding to the condition $a=\alpha$ is
\[
\mathcal L_{\alpha,0}
=
\{(a,b,c,d)\in\mathbb Z^4:a=\alpha\},
\]
and its underlying homogeneous lattice is
\[
\Lambda_{\alpha,0}
=
\{(a,b,c,d)\in\mathbb Z^4:a=0\}.
\]
Under the rescaling
\[
\eta_t(a,b,c,d)=(t^3a,tb,t^{-1}c,t^{-3}d),
\]
the lattice $\eta_t\Lambda_{\alpha,0}$ has orthogonal basis
\[
(0,t,0,0),
\qquad
(0,0,t^{-1},0),
\qquad
(0,0,0,t^{-3}),
\]
and hence
\[
d(\eta_t\Lambda_{\alpha,0})=t^{-3}.
\]
This reversal of the covolume growth is the \textit{source} of escape of mass.

\begin{proposition}[Upper bound for the numerator complement]\label{prop:compact-upper-e1}
For every $T>t_0$ and every $\varepsilon>0$,
\[
M_{(\alpha,0),T}^{(i)}(X)
\ll
X^{3/4}T^3
+
X^{1/2}T^4
+
X^{1/2+\varepsilon}.
\]
\end{proposition}

\begin{proof}
By Lemma~\ref{lem:reducible-fixed-P}, it is enough to count all integral forms with $a=\alpha$, at the cost of $O(X^{1/2+\varepsilon})$. Recall that $\lambda\asymp X^{1/4}$. Since $a=\alpha\neq0$, Lemma~\ref{lem:iwasawa-coefficients} gives
\[
1\ll \frac{\lambda}{t^3},
\qquad\text{and hence}\qquad
t\ll\lambda^{1/3}.
\]
We have that $t\leq T$, so we cover the contributing range by dyadic intervals $S_j\leq t<2S_j$, where $S_j=t_0\,2^j$ for $0\leq j\leq J$, with $S_J\leq T$ and $J\ll\log(X)$. For each $S=S_j$, let $\mathcal M_S(\alpha;X)$ be the set of forms $f\in V_{\mathbb Z}^{(i)}$ admitting a representation $f=n(u)a(t)k\lambda\cdot v^{(i)}$ in the fundamental domain with
\[
S\leq t<2S,
\qquad
c_-X\leq|\Disc(f)|<c_+X,
\qquad
a(f)=\alpha.
\]
The coefficient bounds in this range are
\[
|a|\ll\frac{\lambda}{S^3},
\qquad
|b|\ll\frac{\lambda}{S},
\qquad
|c|\ll\lambda S,
\qquad
|d|\ll\lambda S^3.
\]
Recall that, by Lemma~\ref{constant affects}, we can translate $\mathcal{L}_{\alpha,0}$ to $\Lambda_{\alpha,0}$ and that  
\[
\eta_S(a,b,c,d)
=
(S^3a,Sb,S^{-1}c,S^{-3}d).
\]
Applying the invertible linear map $\eta_S$ to both the above region and the lattice $\Lambda_{\alpha,0}$, we get that the transformed region has side lengths
\[
\ll
\lambda,
\quad
\lambda,
\quad
\lambda ,
\quad
\lambda .
\]
Thus, it is contained in a ball of radius $O(\lambda)$. Moreover, by Lemma~\ref{lem:transformed-lattice-fixed-P}, the
lattice
\[
\eta_S\Lambda_{\alpha,0}
\]
has covolume
$$S^{-3}.$$
After computing all the orthogonal projections  using Lemma \ref{lem:lattice-fixed-P} for the lattice  $\eta_t\Lambda_{\alpha,0}$ and applying Schmidt's theorem, we get\footnote{We use three-dimensional Hausdorff measure to measure
volume in the three-dimensional affine subspace of $\mathbb R^4$;
see \cite{widmer2012lipschitz} for more details. Similarly, the volumes of
the projections appearing in the error term are measured using
Hausdorff measure of the corresponding dimension. For background
on Hausdorff measure, see \cite{evans2025measure}.}
\[
\# M_S(\alpha;X)
\ll \lambda^3S^3+\lambda^2S^4+1.
\]
For $r\in \{3,4\}$, the geometric sum satisfies
\[
\sum_{j=0}^{J}S_j^r
=
t_0^r\frac{2^{r(J+1)}-1}{2^r-1}
\ll S_J^r
\leq T^r.
\]
Thus, summing the dyadic bounds yields
\[
M_{(\alpha,0),T}^{(i)}(X)
\ll
\lambda^3T^3+\lambda^2T^4+\log(2+X)
+X^{1/2+\varepsilon}.
\]
Using $\lambda\asymp X^{1/4}$ and absorbing the logarithmic term gives
\[
M_{(\alpha,0),T}^{(i)}(X)
\ll
X^{3/4}T^3+X^{1/2}T^4+X^{1/2+\varepsilon}.
\]
\end{proof}

As explained in Remark~\ref{averaging method important}, direct
counting in the fundamental domain does not provide a sufficiently
strong lower bound. The following lemma allows us to apply Bhargava's
averaging method when $P=(\pm1,0)$; the same argument does not
extend to other primitive vectors $P$.
\begin{lemma}[Stability of the leading coefficient in the cusp]
\label{lem:leading-coefficient-stability}
Let $\alpha\in\{\pm1\}$ and let $t>X^\theta$, where $\theta>0$. Suppose
\[
f'=\lambda n(u)a(t)k\gamma v,
\]
where $k\in K$, $\gamma$ is sufficiently close to $I_2$, and $v'=\gamma v$ lies in a sufficiently small ball around $v$. Let $f=(f')_{\mathrm{red}}$. Then, for all sufficiently large $X$,
\[
f(\alpha,0)=1
\qquad\Longleftrightarrow\qquad
f'(\alpha,0)=1.
\]
Equivalently,
\[
a(f)=\alpha
\qquad\Longleftrightarrow\qquad
a(f')=\alpha.
\]
\end{lemma}

\begin{proof}
Move $\gamma$ past $k$ and write
\[
k\gamma=\gamma'k,
\qquad
\gamma'=k\gamma k^{-1}.
\]
Since $k\in K$ and $\gamma$ is close to $I_2$, the element $\gamma'$ is also close to $I_2$. Write its Iwasawa decomposition as
\[
\gamma'=n(x)a(s)k'',
\]
where $s$ is close to $1$. Thus
\[
f'
=
\lambda n(u)a(t)n(x)a(s)k''kv.
\]
Using
\[
a(t)n(x)=n(t^2x)a(t),
\]
we get
\[
f'=\lambda n(u+t^2x)a(ts)\widetilde k v,
\qquad
\widetilde k=k''k\in K.
\]
Since $s$ is bounded below by a positive constant and $t>X^\theta$, the new $A$-coordinate $ts$ tends to infinity with $X$. Thus, for sufficiently large $X$, reducing $f'$ only requires translating the lower-unipotent coordinate back into the fixed interval defining the fundamental domain. Hence the reduced form is obtained from $f'$ by the action of some
\[
\ell=n(m)=
\begin{pmatrix}
1&0\\
m&1
\end{pmatrix},
\qquad
m\in\mathbb Z.
\]
Therefore $f=\ell\cdot f'$. Since $\det\ell=1$,
\[
f(\alpha,0)
=
(\ell\cdot f')(\alpha,0)
=
f'((\alpha,0)\ell).
\]
But
\[
(\alpha,0)\ell=(\alpha,0),
\]
so
\[
f(\alpha,0)=f'(\alpha,0).
\]
Finally, because $\alpha=\pm1$, the condition $h(\alpha,0)=1$ is equivalent to $a(h)=\alpha$ for any binary cubic form $h$. This proves the claim.
\end{proof}

\subsubsection*{Reduction to an averaged count}

Fix a sufficiently small ball
\[
B_\varepsilon^{(i)}
=
B(v_i,\varepsilon)\cap V_{\mathbb R}^{(i)}
\]
and write
\[
\mu_i(B_\varepsilon^{(i)})
=
\int_{B_\varepsilon^{(i)}}|\Disc(v)|^{-1}\,dv.
\]
For $\alpha\in\{\pm1\}$, put
\[
S_\alpha^{(i)}
=
\{f\in V_{\mathbb Z}^{(i)}:a(f)=\alpha\}.
\]
Define
\[
\mathcal A_\alpha^{(i)}(X)
=
\frac{1}{\mu_i(B_\varepsilon^{(i)})}
\int_{v'\in B_\varepsilon^{(i)}}
\#\left\{
f'\in \mathcal Fv'
\cap V_{\mathbb Z}^{(i),\mathrm{irr}}:
c_-X\leq|\Disc(f')|<c_+X,\
a(f')=\alpha
\right\}
|\Disc(v')|^{-1}\,dv'.
\]

\begin{proposition}[Reduction to the averaged count]\label{prop:denominator-reduction-e1}
Fix $0<\theta<1/36$. Then
\[
N_{(\alpha,0)}^{(i)}(X)
\gg
\mathcal A_\alpha^{(i)}(X)
-
O\left(
X^{3/4+3\theta}
+
X^{1/2+4\theta}
+
X^{1/2+\varepsilon}
\right).
\]
\end{proposition}

\begin{proof}
Let $v'\in B_\varepsilon^{(i)}$, and let $f=(f')_{\mathrm{red}}$ be the reduced representative of $f'$. By Proposition~\ref{prop:fundamental-multiplicity}, replacing the count over $\mathcal Fv^{(i)}$ by the corresponding count over $\mathcal Fv'$ changes the count by at most a fixed multiplicative constant. Hence
\[
N_{(\alpha,0)}^{(i)}(X)
\gg
\#\left\{
f'\in
\mathcal Fv'
\cap V_{\mathbb Z}^{(i),\mathrm{irr}}:
c_-X\leq|\Disc(f')|<c_+X,\
a(f)=\alpha,\
f=(f')_{\mathrm{red}}
\right\}.
\]

We now restrict to the high-cusp region $\mathcal F_{X^\theta}v'$. Since $\mathcal F_{X^\theta}\subseteq\mathcal F$, we have
\[
N_{(\alpha,0)}^{(i)}(X)
\gg
\#\left\{
f'\in
\mathcal F_{X^\theta}v'
\cap V_{\mathbb Z}^{(i),\mathrm{irr}}:
c_-X\leq|\Disc(f')|<c_+X,\
a(f)=\alpha,\
f=(f')_{\mathrm{red}}
\right\}.
\]
Every form in this count has cusp parameter $t>X^\theta$, so Lemma~\ref{lem:leading-coefficient-stability} applies. Therefore, for all sufficiently large $X$,
\[
a(f)=\alpha
\qquad\Longleftrightarrow\qquad
a(f')=\alpha.
\]
Thus
\[
N_{(\alpha,0)}^{(i)}(X)
\gg
\#\left\{
f'\in
\mathcal F_{X^\theta}v'
\cap V_{\mathbb Z}^{(i),\mathrm{irr}}:
c_-X\leq|\Disc(f')|<c_+X,\
a(f')=\alpha
\right\}.
\]

We write this high-cusp count as the full count minus the contribution from its complement:
\begin{align*}
&\#\left\{
f'\in
\mathcal F_{X^\theta}v'
\cap V_{\mathbb Z}^{(i),\mathrm{irr}}:
c_-X\leq|\Disc(f')|<c_+X,\
a(f')=\alpha
\right\} \\
&\qquad =
\#\left\{
f'\in
\mathcal Fv'
\cap V_{\mathbb Z}^{(i),\mathrm{irr}}:
c_-X\leq|\Disc(f')|<c_+X,\
a(f')=\alpha
\right\} \\
&\qquad\quad -
\#\left\{
f'\in
\mathcal F_{\leq X^\theta}v'
\cap V_{\mathbb Z}^{(i),\mathrm{irr}}:
c_-X\leq|\Disc(f')|<c_+X,\
a(f')=\alpha
\right\}.
\end{align*}
Now Proposition~\ref{prop:compact-upper-e1} applies to the second term. Taking $T=X^\theta$, we get
\[
\#\left\{
f'\in
\mathcal F_{\leq X^\theta}v'
\cap V_{\mathbb Z}^{(i),\mathrm{irr}}:
c_-X\leq|\Disc(f')|<c_+X,\
a(f')=\alpha
\right\}
\ll
X^{3/4+3\theta}
+
X^{1/2+4\theta}
+
X^{1/2+\varepsilon}.
\]
Therefore
\[
N_{(\alpha,0)}^{(i)}(X)
\gg
\#\left\{
f'\in
\mathcal Fv'
\cap V_{\mathbb Z}^{(i),\mathrm{irr}}:
c_-X\leq|\Disc(f')|<c_+X,\
a(f')=\alpha
\right\}
-
O\left(
X^{3/4+3\theta}
+
X^{1/2+4\theta}
+
X^{1/2+\varepsilon}
\right).
\]
Averaging this inequality over $v'\in B_\varepsilon^{(i)}$ with respect to the measure $|\Disc(v')|^{-1}\,dv'$ gives
\[
N_{(\alpha,0)}^{(i)}(X)
\gg
\mathcal A_\alpha^{(i)}(X)
-
O\left(
X^{3/4+3\theta}
+
X^{1/2+4\theta}
+
X^{1/2+\varepsilon}
\right).
\]
\end{proof}

Thus it remains to prove the lower bound
$\mathcal A_\alpha^{(i)}(X)\gg X^{5/6}$.
This is the only point where we use Bhargava's averaging method.

\subsubsection*{Preliminaries for the averaging method}

Let $dv$ denote Euclidean measure on $V_{\mathbb R}$, normalized so that $V_{\mathbb Z}$ has covolume $1$. We use the Haar measure
\[
dg=t^{-2}\,dn\,d^\times t\,dk\,d^\times\lambda
\]
on $\GL_2(\mathbb R)$ in Iwasawa coordinates.

\begin{proposition}[Change of variables]\label{prop:change-variables}
Let $i\in\{0,1\}$, let $\phi\in C_0(V_{\mathbb R}^{(i)})$, and let $v_i\in V_{\mathbb R}^{(i)}$. Then
\[
\int_{\GL_2(\mathbb R)}\phi(g\cdot v_i)\,dg
=
\frac{1}{2\pi}
\int_{\GL_2(\mathbb R)\cdot v_i}
\phi(v)|\Disc(v)|^{-1}\,dv.
\]
Equivalently,
\[
\int_{\GL_2(\mathbb R)}\phi(g\cdot v_i)\,dg
=
\frac{n_i}{2\pi}
\int_{V_{\mathbb R}^{(i)}}\phi(v)|\Disc(v)|^{-1}\,dv,
\]
where $n_i=\#\operatorname{Stab}_{\GL_2(\mathbb R)}(v_i)$.
\end{proposition}

\begin{proof}
Under the twisted action on binary cubic forms,
\[
\Disc(g\cdot v)=(\det g)^2\Disc(v),
\]
while Euclidean measure transforms by
\[
d(g\cdot v)=|\det g|^2\,dv.
\]
Hence
\[
|\Disc(g\cdot v)|^{-1}d(g\cdot v)
=
|\Disc(v)|^{-1}dv.
\]
With the above normalization of Haar measure, the orbit map contributes the factor $1/(2\pi)$. Passing from the real orbit to $V_{\mathbb R}^{(i)}$ introduces the stabilizer multiplicity $n_i$.
\end{proof}

\begin{proposition}[Averaging interchange]\label{prop:averaging-interchange}
Fix $v_i\in V_{\mathbb R}^{(i)}$, and let $\mathcal H^{(i)}\subset\GL_2(\mathbb R)$ be a maximal measurable set such that $\mathcal H^{(i)}v_i=B_\varepsilon^{(i)}$ as an $n_i$-fold multiset. Then, for every $x\in\GL_2(\mathbb R)\cdot v_i$,
\[
\int_{h\in\mathcal H^{(i)}}
\#\{g\in\mathcal F:x=gh\cdot v_i\}\,dh
=
\int_{g\in\mathcal F}
\#\{h\in\mathcal H^{(i)}:x=gh\cdot v_i\}\,dg.
\]
\end{proposition}

\begin{proof}
Choose $\gamma_x\in\GL_2(\mathbb R)$ such that $x=\gamma_x\cdot v_i$, and let $S_i=\operatorname{Stab}_{\GL_2(\mathbb R)}(v_i)$. For fixed $h$, the equation $x=gh\cdot v_i$ is equivalent to
\[
g=\gamma_x\sigma h^{-1}
\]
for some $\sigma\in S_i$. Therefore
\[
\#\{g\in\mathcal F:x=gh\cdot v_i\}
=
\sum_{\sigma\in S_i}
\mathbf 1_{\mathcal F}(\gamma_x\sigma h^{-1}).
\]
Integrating in $h$ and making the change of variables $g=\gamma_x\sigma h^{-1}$, using the unimodularity of $\GL_2(\mathbb R)$, gives
\[
\int_{h\in\mathcal H^{(i)}}
\mathbf 1_{\mathcal F}(\gamma_x\sigma h^{-1})\,dh
=
\int_{g\in\mathcal F}
\mathbf 1_{\mathcal H^{(i)}}(g^{-1}\gamma_x\sigma)\,dg.
\]
Summing over $\sigma\in S_i$ gives the stated identity.
\end{proof}

\begin{proposition}[Averaging identity]\label{prop:averaging-identity-e1}
For $\alpha\in \{\pm1 \}$ we have
\[
\mathcal A_\alpha^{(i)}(X)
=
\frac{2\pi}{\mu_i(B_\varepsilon^{(i)})}
\int_{g\in\mathcal F}
\#\left\{
f'\in S_\alpha^{(i),\mathrm{irr}}
\cap gB_\varepsilon^{(i)}:
c_-X\leq|\Disc(f')|<c_+X
\right\}\,dg.
\]
where 
$S_\alpha^{(i)}
=
\{f\in V_{\mathbb Z}^{(i)}:a(f)=\alpha\}.$
\end{proposition}

\begin{proof}
Expanding the counting function and interchanging the finite sum with the integral gives
\[
\mu_i(B_\varepsilon^{(i)})\mathcal A_\alpha^{(i)}(X)
=
\sum_{\substack{
f'\in S_\alpha^{(i),\mathrm{irr}}\\
c_-X\leq|\Disc(f')|<c_+X
}}
\int_{v'\in B_\varepsilon^{(i)}}
\#\{g\in\mathcal F:f'=g\cdot v'\}
|\Disc(v')|^{-1}\,dv'.
\]
Using Proposition~\ref{prop:change-variables}, the inner integral becomes
\[
\frac{2\pi}{n_i}
\int_{h\in\mathcal H^{(i)}}
\#\{g\in\mathcal F:f'=gh\cdot v_i\}\,dh.
\]
Applying Proposition~\ref{prop:averaging-interchange}, we obtain
\[
\mu_i(B_\varepsilon^{(i)})\mathcal A_\alpha^{(i)}(X)
=
\frac{2\pi}{n_i}
\int_{g\in\mathcal F}
\#\left\{
f'\in S_\alpha^{(i),\mathrm{irr}}
\cap g\mathcal H^{(i)}v_i:
c_-X\leq|\Disc(f')|<c_+X
\right\}\,dg.
\]
Since $\mathcal H^{(i)}v_i$ is an $n_i$-fold cover of $B_\varepsilon^{(i)}$, the factor $n_i$ cancels. Therefore
\[
\mathcal A_\alpha^{(i)}(X)
=
\frac{2\pi}{\mu_i(B_\varepsilon^{(i)})}
\int_{g\in\mathcal F}
\#\left\{
f'\in S_\alpha^{(i),\mathrm{irr}}
\cap gB_\varepsilon^{(i)}:
c_-X\leq|\Disc(f')|<c_+X
\right\}\,dg.
\]
\end{proof}

\begin{proposition}[Averaging lower bound]\label{prop:averaging-lower-e1}
For $\alpha\in\{\pm1\}$,
\[
\mathcal A_\alpha^{(i)}(X)\gg X^{5/6}.
\]
\end{proposition}

\begin{proof}
By Proposition~\ref{prop:averaging-identity-e1},
\[
\mathcal A_\alpha^{(i)}(X)
=
\frac{2\pi}{\mu_i(B_\varepsilon^{(i)})}
\int_{g\in\mathcal F}
\#\left\{
f'\in S_\alpha^{(i),\mathrm{irr}}\cap gB_\varepsilon^{(i)}:
c_-X\leq|\Disc(f')|<c_+X
\right\}\,dg.
\]
By Lemma~\ref{lem:reducible-fixed-P}, it is enough to count \textit{all} integral forms with $a=\alpha$, at the cost of $O(X^{1/2+\varepsilon})$. Recall that $\lambda\asymp X^{1/4}$. Since $a=\alpha\neq0$, Lemma~\ref{lem:iwasawa-coefficients} gives
\[
1\ll \frac{\lambda}{t^3},
\qquad\text{and hence}\qquad
t\ll\lambda^{1/3}.
\]

Write $g=n(u)a(t)k\lambda$. We can restrict the integral to a positive-measure subregion. Choose fixed positive-measure subsets $N_0\subset N$ and $K_0\subset K$, constants $0<c_1<c_2$, and a sufficiently small $c>0$ such that, whenever
\[
u\in N_0,\qquad k\in K_0,\qquad
c_1X^{1/4}\leq\lambda\leq c_2X^{1/4},
\qquad
t_*\leq t\leq c\lambda^{1/3},
\]
the affine slice $a=\alpha$ meets $n(u)a(t)k\lambda B_\varepsilon^{(i)}$ away from the boundary. Here
\[
t_*=\frac{\sqrt[4]{3}}{\sqrt2}.
\]

Inside $n(u)a(t)k\lambda B_\varepsilon^{(i)}$, we use continuity of the discriminant and we can choose a compact subregion
\[
\mathcal C_{u,t,\lambda,k}
\subset
n(u)a(t)k\lambda B_\varepsilon^{(i)}
\]
whose side lengths are fixed positive multiples of
\[
\frac{\lambda}{t^3},
\qquad
\frac{\lambda}{t},
\qquad
\lambda t,
\qquad
\lambda t^3.
\]
After translating $\mathcal L_{\alpha,0}$ to $\Lambda_{\alpha,0}$ due to Remark~\ref{constant affects} and applying $\eta_{t}$ to both $\mathcal C_{u,t,\lambda,k}$ and $\Lambda_{\alpha,0}$ we get that  the coefficient coordinates of forms in $\eta_t \mathcal C_{u,t,\lambda,k}$ have side lengths comparable to
\[
\lambda, \qquad \lambda,\qquad\lambda ,\qquad\lambda,
\]
uniformly for $u\in N_0$ and $k\in K_0$.
Thus, the region $\eta_t \mathcal C_{u,t,\lambda,k}$ contains a ball of radius $O(\lambda)$. Additionally, by Lemma~\ref{lem:transformed-lattice-fixed-P}
\[
d(\eta_t\Lambda_{\alpha,0})=t^{-3},
\]
After computing all the orthogonal projections  using Lemma~\ref{lem:lattice-fixed-P} for the lattice  $\eta_t\Lambda_{\alpha,0}$ and applying Schmidt's theorem~\ref{thm:schmidt}, we get\footnote{We use three-dimensional Hausdorff measure to measure
volume in the three-dimensional affine subspace of $\mathbb R^4$;
see \cite{widmer2012lipschitz} for more details. Similarly, the volumes of
the projections appearing in the error term are measured using
Hausdorff measure of the corresponding dimension. For background
on Hausdorff measure, see \cite{evans2025measure}.}
\[
\#(\mathcal C_{u,t,\lambda,k}\cap\mathcal L_{\alpha,0})
\geq
c_0\lambda^3t^3
-
M(\lambda^2t^4+1),
\]
with constants $c_0,M>0$ independent of $u,t,\lambda,k$ in the chosen range.
Restricting the averaging integral to this range gives
\begin{align*}
\mathcal A_\alpha^{(i)}(X)
&\gg
\int_{c_1X^{1/4}}^{c_2X^{1/4}}
\int_{t_*}^{c\lambda^{1/3}}
\left(c_0\lambda^3t^3-M(\lambda^2t^4+1)\right)
t^{-3}\lambda^{-1}\,dt\,d\lambda\\
&\qquad
-
O(X^{1/2+\varepsilon}),
\end{align*}
where we used $t^{-2}d^\times t\,d^\times\lambda=t^{-3}\lambda^{-1}\,dt\,d\lambda$. The main term is
\[
\int_{c_1X^{1/4}}^{c_2X^{1/4}}
\int_{t_*}^{c\lambda^{1/3}}
\lambda^3t^3t^{-3}\lambda^{-1}\,dt\,d\lambda
\asymp X^{5/6}.
\]
The two error integrals are $O(X^{2/3})$ and $O(1)$, respectively. The reducible contribution is $O(X^{1/2+\varepsilon})$, which is $o(X^{5/6})$ after choosing $\varepsilon<1/3$. Hence
\[
\mathcal A_\alpha^{(i)}(X)\gg X^{5/6}.
\]
\end{proof}

\begin{proposition}[Denominator lower bound]\label{prop:denominator-e1}
For $\alpha\in\{\pm1\}$,
\[
N_{(\alpha,0)}^{(i)}(X)\gg X^{5/6}.
\]
\end{proposition}

\begin{proof}
By Proposition~\ref{prop:denominator-reduction-e1},
\[
N_{(\alpha,0)}^{(i)}(X)
\gg
\mathcal A_\alpha^{(i)}(X)
-
O\left(
X^{3/4+3\theta}
+
X^{1/2+4\theta}
+
X^{1/2+\varepsilon}
\right).
\]
By Proposition~\ref{prop:averaging-lower-e1},
\[
\mathcal A_\alpha^{(i)}(X)\gg X^{5/6}.
\]
Since $0<\theta<1/36$, we have
\[
\frac34+3\theta<\frac56,
\qquad
\frac12+4\theta<\frac56.
\]
Choosing $\varepsilon<1/3$, all error terms are $o(X^{5/6})$. Therefore
\[
N_{(\alpha,0)}^{(i)}(X)\gg X^{5/6}.
\]
\end{proof}

We can now prove Theorem~\ref{thm:escape-e1}.

\begin{proof}
Let $P=(\alpha,0)$. Since
\[
N_P^{(i)}(X)
=
N_{P,T}^{(i)}(X)
+
M_{P,T}^{(i)}(X),
\]
we have
\[
\left|
\frac{N_{P,T}^{(i)}(X)}{N_P^{(i)}(X)}
-1
\right|
=
\frac{M_{P,T}^{(i)}(X)}{N_P^{(i)}(X)}.
\]
By Propositions~\ref{prop:compact-upper-e1} and~\ref{prop:denominator-e1}, for every fixed $T$,
\[
\frac{M_{P,T}^{(i)}(X)}{N_P^{(i)}(X)}
\ll
T^3X^{-1/12}
+
T^4X^{-1/3}
+
X^{-1/3+\varepsilon}.
\]
The right-hand side tends to $0$ as $X\to\infty$ for every fixed $T$. Hence
\[
\limsup_{X\to\infty}
\left|
\frac{N_{P,T}^{(i)}(X)}{N_P^{(i)}(X)}
-1
\right|
=0.
\]
Taking $T\to\infty$ gives
\[
\lim_{T\to\infty}
\limsup_{X\to\infty}
\left|
\frac{N_{P,T}^{(i)}(X)}{N_P^{(i)}(X)}
-1
\right|
=0.
\]
\end{proof}

\subsection{Averaging over growing solutions}

The estimates used for Theorem~\ref{thm:non-escape-fixed-P} are sufficient to average over a growing family of primitive solutions.

\begin{theorem}\label{thm:non-escape-averaged}
Fix $\delta>0$, and let $Y=Y(X)$ satisfy
\[
2\leq Y\leq X^{1/12-\delta}.
\]
Let
\[
\mathcal P(Y)
=
\{(\alpha,\beta)\in\mathbb Z_{\mathrm{prim}}^2:
Y\leq|\alpha|,|\beta|<2Y\}.
\]
Define
\[
J_T^{(i)}(X;Y)
=
\sum_{P\in\mathcal P(Y)}N_{P,T}^{(i)}(X),
\qquad
J^{(i)}(X;Y)
=
\sum_{P\in\mathcal P(Y)}N_P^{(i)}(X).
\]
Then
\[
\lim_{T\to\infty}\limsup_{X\to\infty}
\frac{J_T^{(i)}(X;Y)}{J^{(i)}(X;Y)}=0.
\]
\end{theorem}

\begin{proof}
Every $P=(\alpha,\beta)\in\mathcal P(Y)$ has $\beta\neq0$, so we can apply Propositions~\ref{prop:cusp-upper-fixed-P} and~\ref{prop:compact-lower-fixed-P}. For every fixed $\varepsilon>0$,
\[
N_{P,T}^{(i)}(X)
\ll
\frac{X^{3/4}}{T^3|\beta|^3}
+
\frac{X^{1/2}}{|\beta|^2}
+
X^{1/2+\varepsilon}.
\]
Since there are $O(Y)$ choices for $\alpha$,
\[
\sum_{P\in\mathcal P(Y)}\frac{1}{|\beta|^3}
\ll\frac1Y,
\qquad
\sum_{P\in\mathcal P(Y)}\frac{1}{|\beta|^2}
\ll1,
\qquad
\#\mathcal P(Y)\ll Y^2.
\]
Hence
\[
J_T^{(i)}(X;Y)
\ll
\frac{X^{3/4}}{T^3Y}
+
X^{1/2}
+
X^{1/2+\varepsilon}Y^2.
\]

For the denominator, Proposition~\ref{prop:compact-lower-fixed-P} gives
\[
N_P^{(i)}(X)
\geq
c\frac{X^{3/4}}{(|\alpha|+|\beta|)^3}
-
C\frac{X^{1/2}}{|\beta|^2}
-
C_\varepsilon X^{1/2+\varepsilon}.
\]
Since $2Y\leq|\alpha|+|\beta|<4Y$ and $\#\mathcal P(Y)\asymp Y^2$,
\[
\sum_{P\in\mathcal P(Y)}
\frac{1}{(|\alpha|+|\beta|)^3}
\asymp\frac1Y.
\]
Therefore
\[
J^{(i)}(X;Y)
\geq
c_1\frac{X^{3/4}}{Y}
-
C_1X^{1/2}
-
C_{1,\varepsilon}X^{1/2+\varepsilon}Y^2,
\]
where all the above constants are positive.
Choose $0<\varepsilon<3\delta$. Since $Y\leq X^{1/12-\delta}$,
\[
\frac{X^{1/2+\varepsilon}Y^2}{X^{3/4}/Y}
=
X^{-1/4+\varepsilon}Y^3
\leq
X^{\varepsilon-3\delta}
\longrightarrow0,
\]
and
\[
\frac{X^{1/2}}{X^{3/4}/Y}
=YX^{-1/4}
\longrightarrow0.
\]
Thus
\[
J^{(i)}(X;Y)
\gg
\frac{X^{3/4}}{Y},
\]
while
\[
J_T^{(i)}(X;Y)
\ll
\frac{X^{3/4}}{T^3Y}
+
o\left(\frac{X^{3/4}}{Y}\right).
\]
Consequently,
\[
\limsup_{X\to\infty}
\frac{J_T^{(i)}(X;Y)}{J^{(i)}(X;Y)}
\ll T^{-3}.
\]
Letting $T\to\infty$ proves the theorem.
\end{proof}

The exponent $1/12$ in the preceding theorem arises from
the range in which the counting error is smaller than
the main contribution. Taking $T=t_0$ in the dyadic
counting argument gives two-sided bounds for the averaged
count $J^{(i)}(X;Y)$, uniformly for
$2\leq Y\leq X^{1/12-\delta}$. This count enumerates pairs
$(f,P)$, where $f$ is a reduced irreducible integral binary
cubic form and $P=(\alpha,\beta)$ satisfies $f(P)=1$ and
$Y\leq|\alpha|,|\beta|<2Y$. We record this consequence
in the following corollary.

\begin{corollary}[Asymptotic bounds for the averaged count]
\label{cor:averaged-count-order}
Fix $0<\delta<1/12$ and $i\in\{0,1\}$. Uniformly for
$2\leq Y\leq X^{1/12-\delta}$, we have
\[
J^{(i)}(X;Y)\asymp\frac{X^{3/4}}{Y},
\]
where the implied constants may depend on $\delta$ and
the fixed counting parameters, but not on $X$ or $Y$.
\end{corollary}

\begin{proof}
The lower bound follows from the proof of
Theorem~\ref{thm:non-escape-averaged}. For the upper bound,
we sum over the dyadic intervals starting at
$t_0=\sqrt[4]{3}/\sqrt2$, as in the preceding counting
arguments, to obtain
\[
J^{(i)}(X;Y)
\ll \frac{X^{3/4}}{Y}+Y^2X^{1/2+\varepsilon}.
\]
Choose $0<\varepsilon<3\delta$. Since
$Y\leq X^{1/12-\delta}$,
\[
Y^2X^{1/2+\varepsilon}
\leq \frac{X^{3/4}}{Y}\,X^{-3\delta+\varepsilon}
=o\left(\frac{X^{3/4}}{Y}\right)
\]
uniformly in the stated range. This proves the upper bound.
\end{proof}

The averaging in Theorem~\ref{thm:non-escape-averaged} is restricted to solutions
$P\in\mathcal P(Y)$ with $Y\le X^{1/12-\delta}$. It is natural to ask whether the
non-escape phenomenon persists with no restriction on the size of $P$ at all.
Before stating this as a conjecture, we record that the relevant counting
function is well defined, so that the
statement is not vacuous.

For $i\in\{0,1\}$, define
\[
J^{(i)}(X)
=
\sum_{P\in\mathcal P}N_P^{(i)}(X),
\qquad
J_T^{(i)}(X)
=
\sum_{P\in\mathcal P}N_{P,T}^{(i)}(X).
\]

\begin{proposition}[Finiteness and order of $J^{(i)}(X)$]
\label{prop:J-order}
For every $i\in\{0,1\}$,
\[
X^{3/4}
\ll_i
J^{(i)}(X)
\ll_i
X.
\]
In particular, $J^{(i)}(X)$ is finite for every $X$, and the ratio
$J_T^{(i)}(X)/J^{(i)}(X)$ is well defined for all sufficiently large $X$.
\end{proposition}

\begin{proof}
\emph{Lower bound.} Fix any single $P_0\in\mathcal P$. Then
$J^{(i)}(X)\ge N_{P_0}^{(i)}(X)$, and Proposition~\ref{prop:compact-lower-fixed-P}
gives $N_{P_0}^{(i)}(X)\gg_{P_0} X^{3/4}$.

\emph{Upper bound.} Since every term in the sum defining $J^{(i)}(X)$ is a
nonnegative integer, we may exchange the order of summation and write
\[
J^{(i)}(X)
=
\sum_{\substack{f\in\mathcal Fv^{(i)}\cap V_Z^{(i),\mathrm{irr}}\\
c_-X\le|\mathrm{Disc}(f)|<c_+X}}
\#\{P\in\mathcal P: f(P)=1\}.
\]

Using the bounds of Delone, Evertse and Nagell \cite{Evertse1983,delaunay1930darstellung,nagell1928darstellung}, $\#\{P\in\mathbb Z^2:f(P)=1\}\le B_i$ for
every irreducible $f$ of nonzero discriminant, where $B_0=12$ and $B_1=5$.
Hence
\[
J^{(i)}(X)
\le
B_i\cdot
\#\{f\in\mathcal Fv^{(i)}\cap V_Z^{(i),\mathrm{irr}}:|\mathrm{Disc}(f)|<c_+X\}.
\]
The fundamental multiset
$\mathcal Fv^{(i)}$ represents every $\GL_2(\mathbb Z)$-orbit with multiplicity
at most $n_i$, so
\[
J^{(i)}(X)
\le
B_i\,n_i\,N(V_Z^{(i)};c_+X)
=
O_i(X)
\]
by Theorem~\ref{thm:bst-davenport}.
\end{proof}

Proposition~\ref{prop:J-order} gives
$X^{3/4}\ll J^{(i)}(X)\ll X$. However,
Theorem~\ref{thm:non-escape-averaged} and
Corollary~\ref{cor:averaged-count-order} only treat
$2\leq Y\leq X^{1/12-\delta}$. We conjecture that non-escape
continues to hold in this unrestricted setting. The numerical experiments in Appendix~\ref{app:thin-families}
provide further motivation for this conjecture.

\begin{conjecture}[Non-escape without a height restriction]
\label{conj:non-escape-unrestricted}
Let $i\in\{0,1\}$ and put
\[
\mathcal P
=\mathbb Z_{\mathrm{prim}}^2\setminus\{(1,0),(-1,0)\}.
\]
Define
\[
J^{(i)}(X)=\sum_{P\in\mathcal P}N_P^{(i)}(X),
\qquad
J_T^{(i)}(X)=\sum_{P\in\mathcal P}N_{P,T}^{(i)}(X).
\]
Then
\[
\lim_{T\to\infty}\limsup_{X\to\infty}
\frac{J_T^{(i)}(X)}{J^{(i)}(X)}=0.
\]
\end{conjecture}

\section{Normalized Weierstrass models and marked reduced binary cubic forms}
\label{Correspondence}

In this section, we relate normalized integral Weierstrass models to reduced binary cubic forms equipped with a distinguished integral solution to the Thue equation $F(X,Y)=1$. Starting from a normalized model, a choice of reduction matrix produces a reduced binary cubic form together with a primitive vector $P$ satisfying $f(P)=1$. Conversely, a reduced binary cubic form together with such a vector determines a unique normalized integral Weierstrass model and a reduction matrix. We show that these constructions are inverse to one another.

After fixing the marked solution $P$, this correspondence gives
a bijection between the chosen reduced representatives satisfying
$f(P)=1$ and normalized integral Weierstrass models admitting a
reduction matrix $U$ such that $(1,0)U^{-1}=P$.
To transfer the mass estimates from the preceding section, we
also consider the models obtained from its fundamental multiset.
The discriminant identity $\Delta(E)=16\Disc(f)$ and a uniform
bound on the number of preimages allow us to transfer escape
and non-escape to these families. Finally, we control the
multiplicity arising when the marked solution is forgotten.
We begin with two elementary normalization lemmas.

\begin{lemma}\label{lem:monic-from-solution}
Let $P=(\alpha,\beta)\in\mathbb Z^2$ be primitive, and let
\[
F(X,Y)=aX^3+bX^2Y+cXY^2+dY^3
\]
be an integral binary cubic form satisfying
$F(\alpha,\beta)=1.$
Then there exists $M\in\SL_2(\mathbb Z)$ such that $M\cdot F$ is monic.
\end{lemma}

\begin{proof}
Since $\gcd(\alpha,\beta)=1$, there exist $\gamma,\delta\in\mathbb Z$ such that
$\alpha\delta-\beta\gamma=1.$
Set
$M=
\begin{pmatrix}
\alpha&\beta\\
\gamma&\delta
\end{pmatrix}
\in\SL_2(\mathbb Z).$
The leading coefficient of $M\cdot F$ is
$(M\cdot F)(1,0)
=
F((1,0)M)
=
F(\alpha,\beta)
=
1.$
Thus $M\cdot F$ is monic.
\end{proof}

\begin{lemma}\label{lem:b-mod-3}
Let
\[
F(X,Y)
=
X^3+bX^2Y+cXY^2+dY^3
\in\mathbb Z[X,Y].
\]
Then there exists $T_k\in\SL_2(\mathbb Z)$ such that $T_k\cdot F$ is monic and its coefficient of $X^2Y$ belongs to $\{0,1,2\}$.
\end{lemma}

\begin{proof}
Let $u\in\{0,1,2\}$ be the unique integer satisfying
$b\equiv u\pmod 3,$
and set
$k=\frac{u-b}{3}\in\mathbb Z.$
For
$T_k=
\begin{pmatrix}
1&0\\
k&1
\end{pmatrix},$
the action of $T_k$ amounts to replacing $X$ by $X+kY$. Hence the coefficient of $X^2Y$ in $T_k\cdot F$ is
$b+3k=u,$
while the leading coefficient remains equal to $1$.
\end{proof}

We define
\[
\mathcal K^{(i)}
=
\left\{
X^3+uX^2Y+AXY^2+BY^3:
u\in\{0,1,2\},\ A,B\in\mathbb Z
\right\}
\cap V_{\mathbb Z}^{(i)}.
\]
When irreducibility is required, we write
\[
\mathcal K_{\mathrm{irr}}^{(i)}
=
\{F\in\mathcal K^{(i)}:
F\text{ is irreducible over }\mathbb Q\}.
\]

Every form
\[
F_E(X,Y)
=
X^3+uX^2Y+AXY^2+BY^3
\in\mathcal K_{\mathrm{irr}}^{(i)}
\]
determines a normalized integral Weierstrass model \cite{silverman2009arithmetic}
\[
E:\ y^2=x^3+ux^2+Ax+B,
\qquad
u\in\{0,1,2\}.\]
Conversely, every normalized model of this form whose defining
cubic is irreducible over $\mathbb Q$ determines an element
$F_E\in\mathcal K_{\mathrm{irr}}^{(i)}$, where $i$ records the
sign of its discriminant. Throughout this section, we count
normalized equations themselves, rather than isomorphism
classes of elliptic curves.

Recall that an irreducible binary cubic form $f$ is reduced
if its shape parameter belongs to the standard modular
fundamental domain
$\mathfrak F=\{\tau\in\mathfrak H:
|\tau|\geq1,\ -\frac12\leq\Re(\tau)\leq\frac12\}$.
We reserve $\mathcal F$ for Gauss's usual fundamental domain
for $\mathrm{GL}_2(\mathbb Z)\backslash\mathrm{GL}_2(\mathbb R)$
in $\mathrm{GL}_2(\mathbb R)$.

Choose once and for all one reduced representative in each
$\SL_2(\mathbb Z)$-orbit of irreducible integral binary cubic
forms of discriminant sign $i$. We denote the resulting set of representatives by
$\mathcal R^{(i)}$. Let $E$ be a normalized integral Weierstrass model with
$F_E\in\mathcal K_{\mathrm{irr}}^{(i)}$,
and let $U\in\SL_2(\mathbb Z)$ satisfy
$f_E:=U\cdot F_E\in\mathcal R^{(i)}.$
We associate to the reduction matrix $U$ the vector
$P_E=(1,0)U^{-1}.$

\begin{lemma}\label{lem:distinguished-solution}
The vector $P_E$ is primitive and integral, and it satisfies
$f_E(P_E)=1.$
\end{lemma}

\begin{proof}
Since $U^{-1}\in\SL_2(\mathbb Z)$, the vector
$P_E=(1,0)U^{-1}$ is primitive and integral. Moreover,
\[
f_E(P_E)
=(U\cdot F_E)((1,0)U^{-1})
=F_E((1,0)U^{-1}U)
=F_E(1,0)=1.
\]
\end{proof}

Thus a normalized integral Weierstrass model together with a choice of reduction matrix determines a marked reduced binary cubic form:
\[
(E,U)
\longmapsto
(f_E,P_E),
\qquad
f_E(P_E)=1.
\]

We now construct the inverse correspondence. Let
$f\in\mathcal R^{(i)}$ and let $P=(\alpha,\beta)\in\mathbb Z^2$
satisfy $f(P)=1$. The vector $P$ is automatically primitive.
Indeed, if $m=\gcd(\alpha,\beta)$, then the homogeneity of $f$
gives $m^3\mid f(\alpha,\beta)=1$, and hence $m=1$.

Choose a matrix
\[
M_P=\begin{pmatrix}\alpha&\beta\\\gamma&\delta\end{pmatrix}
\in\SL_2(\mathbb Z).
\]
By Lemma~\ref{lem:monic-from-solution}, the form
$G_{f,P}:=M_P\cdot f$ is monic. Write
\[
G_{f,P}(X,Y)=X^3+bX^2Y+cXY^2+dY^3.
\]
Let $u\in\{0,1,2\}$ be the unique integer satisfying
$b\equiv u\pmod 3$, and select $k\in \mathbb{Z}$ as in Lemma~\ref{lem:b-mod-3}. Define
\[
T_k=\begin{pmatrix}1&0\\k&1\end{pmatrix},
\qquad
F_{f,P}=T_k\cdot G_{f,P}.
\]
Then $F_{f,P}(X,Y)=X^3+uX^2Y+AXY^2+BY^3$ for some
$A,B\in\mathbb Z$. We associate to $(f,P)$ the normalized
integral Weierstrass model
\[
E_{f,P}:\ y^2=x^3+ux^2+Ax+B.
\]

We also define $U_{f,P}=M_P^{-1}T_{-k}$. By construction,
\[
U_{f,P}\cdot F_{f,P}
=M_P^{-1}T_{-k}\cdot\bigl(T_kM_P\cdot f\bigr)
=f.
\]
Furthermore,
\[
(1,0)U_{f,P}^{-1}
=(1,0)T_kM_P
=(1,0)M_P
=P.
\]

The construction does not depend on the choice of the second row of $M_P$.

\begin{lemma}\label{lem:independent-completion}
The form $F_{f,P}$ and the matrix $U_{f,P}$ are independent of the choice of the matrix $M_P\in\SL_2(\mathbb Z)$ whose first row is $P$.
\end{lemma}

\begin{proof}
Let $M_P$ and $M_P'$ be two matrices in $\SL_2(\mathbb Z)$
having first row $P$. Then $M_P'=T_mM_P$ for some
$m\in\mathbb Z$, where
\[
T_m=\begin{pmatrix}1&0\\m&1\end{pmatrix}.
\]
Hence $M_P'\cdot f=T_m\cdot(M_P\cdot f)$.

Let $b$ be the coefficient of $X^2Y$ in $M_P\cdot f$.
The corresponding coefficient in $M_P'\cdot f$ is
$b'=b+3m$. Thus $b$ and $b'$ determine the same
representative $u\in\{0,1,2\}$ modulo $3$. Setting
$k=(u-b)/3$ and $k'=(u-b')/3$, we obtain $k'=k-m$.
It follows that
\[
T_{k'}\cdot(M_P'\cdot f)
=T_{k-m}T_m\cdot(M_P\cdot f)
=T_k\cdot(M_P\cdot f).
\]
Hence the normalized form $F_{f,P}$ is independent of
the choice of $M_P$. Moreover,
\[
(M_P')^{-1}T_{-k'}
=M_P^{-1}T_{-m}T_{-(k-m)}
=M_P^{-1}T_{-k}.
\]
Thus $U_{f,P}$ is also independent of this choice.
\end{proof}

The normalization construction and its independence of the
completion $M_P$ use only the condition $f(P)=1$.
Thus $F_{f,P}$, $E_{f,P}$, and $U_{f,P}$ are defined for
every irreducible integral binary cubic form $f$ equipped
with such a solution, whether or not $f$ belongs to the
chosen set $\mathcal R^{(i)}$.

Define
\[
\widetilde{\mathcal R}^{(i)}
=
\left\{
(f,P):
f\in\mathcal R^{(i)},\
P\in\mathbb Z^2,\
f(P)=1
\right\},
\]
and
\[
\widetilde{\mathcal E}^{(i)}
=
\left\{
(E,U):
F_E\in\mathcal K_{\mathrm{irr}}^{(i)},\
U\in\SL_2(\mathbb Z),\
U\cdot F_E\in\mathcal R^{(i)}
\right\}.
\]

\begin{proposition}\label{prop:direct-correspondence}
The map
\[
\Phi:
\widetilde{\mathcal E}^{(i)}
\longrightarrow
\widetilde{\mathcal R}^{(i)}
\]
defined by
\[
\Phi(E,U)
=
\bigl(U\cdot F_E,(1,0)U^{-1}\bigr)
\]
is a bijection. Its inverse is given by
\[
(f,P)
\longmapsto
(E_{f,P},U_{f,P}).
\]
\end{proposition}

\begin{proof}
Let $(E,U)\in\widetilde{\mathcal E}^{(i)}$ and set
$f=U\cdot F_E$ and $P=(1,0)U^{-1}$. In the inverse
construction, we may choose $M_P=U^{-1}$, since its
first row is $P$. Then
\[
M_P\cdot f=U^{-1}\cdot(U\cdot F_E)=F_E.
\]
The coefficient of $X^2Y$ in $F_E$ already belongs to
$\{0,1,2\}$. Hence the normalization parameter is $k=0$,
and the inverse construction recovers both $E$ and $U$.

Conversely, starting with
$(f,P)\in\widetilde{\mathcal R}^{(i)}$, the construction gives
$U_{f,P}\cdot F_{f,P}=f$ and $(1,0)U_{f,P}^{-1}=P$.
Therefore, $\Phi(E_{f,P},U_{f,P})=(f,P)$.
\end{proof}

We now fix a primitive vector
\[
P=(\alpha,\beta)\in\mathbb Z^2
\]
and define
\[
\mathcal R_P^{(i)}
=
\{f\in\mathcal R^{(i)}:f(P)=1\}.
\]
On the Weierstrass-model side, define
\[
\mathcal E_P^{(i)}
=
\left\{
E:
\begin{array}{l}
F_E\in\mathcal K_{\mathrm{irr}}^{(i)},\ \text{and there exists }
U\in\SL_2(\mathbb Z)\\[2pt]
\text{such that }
U\cdot F_E\in\mathcal R^{(i)}
\text{ and }
(1,0)U^{-1}=P
\end{array}
\right\}.
\]

\begin{corollary}\label{cor:fixed-P-correspondence}
For every primitive vector $P$, the map
\[
f\longmapsto E_{f,P}
\]
is a bijection
\[
\mathcal R_P^{(i)}
\longrightarrow
\mathcal E_P^{(i)}.
\]
\end{corollary}

\begin{proof}
Surjectivity follows from Proposition~\ref{prop:direct-correspondence}.

To prove injectivity, suppose that $E_{f_1,P}=E_{f_2,P}=E$
for some $f_1,f_2\in\mathcal R_P^{(i)}$. By construction, there
exist $U_1,U_2\in\SL_2(\mathbb Z)$ such that
$f_1=U_1\cdot F_E$ and $f_2=U_2\cdot F_E$. Thus $f_1$ and $f_2$
lie in the same $\SL_2(\mathbb Z)$-orbit. Since $\mathcal R^{(i)}$
contains exactly one chosen representative from each orbit,
we obtain $f_1=f_2$.

It remains to verify that, once $E$, $f$, and $P$ are fixed,
the matrix $U$ is uniquely determined. Suppose that
$U_1\cdot F_E=U_2\cdot F_E=f$ and
$(1,0)U_1^{-1}=(1,0)U_2^{-1}=P$.
Set $G=U_2^{-1}U_1$. Then $G\cdot F_E=F_E$ and
$(1,0)G^{-1}=(1,0)$. Consequently,
\[
G=\begin{pmatrix}1&0\\m&1\end{pmatrix}
\qquad\text{for some }m\in\mathbb Z.
\]
Let $u\in\{0,1,2\}$ be the coefficient of $X^2Y$ in $F_E$.
The coefficient of $X^2Y$ in $G\cdot F_E$ is $u+3m$.
Since $G\cdot F_E=F_E$, we obtain $m=0$.
Hence $G=I_2$ and $U_1=U_2$.
\end{proof}

For
\[
E\in\mathcal E_P^{(i)},
\]
let
\[
f_{E,P}\in\mathcal R_P^{(i)}
\]
denote the unique reduced marked form corresponding to $E$ under Corollary~\ref{cor:fixed-P-correspondence}.

We now transfer the mass estimates using the same fundamental
multiset as in the preceding section. Fix a primitive vector
$P$, and retain the constants $0<c_-<c_+$ used there. Let $\mathscr S_P^{(i)}(X)$ be the multiset of forms
$f\in\mathcal Fv^{(i)}\cap V_{\mathbb Z}^{(i),\mathrm{irr}}$
such that $f(P)=1$ and
$c_-X\leq|\Disc(f)|<c_+X$.
Let $\mathscr S_{P,T}^{(i)}(X)$ be its cuspidal submultiset,
obtained by requiring $f\in\mathcal F_Tv^{(i)}$.
Thus
\[
\#\mathscr S_P^{(i)}(X)=N_P^{(i)}(X),
\qquad
\#\mathscr S_{P,T}^{(i)}(X)=N_{P,T}^{(i)}(X).
\]

Define the sets of distinct normalized models
\[
\mathscr E_P^{(i)}(X)
=
\{E_{f,P}:f\in\mathscr S_P^{(i)}(X)\},
\qquad
\mathscr E_{P,T}^{(i)}(X)
=
\{E_{f,P}:f\in\mathscr S_{P,T}^{(i)}(X)\},
\]
and put
$C_P^{(i)}(X)=\#\mathscr E_P^{(i)}(X)$ and
$C_{P,T}^{(i)}(X)=\#\mathscr E_{P,T}^{(i)}(X)$.
Here membership in $\mathscr E_{P,T}^{(i)}(X)$ means that
the model arises from at least one cuspidal representative
satisfying $f(P)=1$.

Since $F_{f,P}$ is $\SL_2(\mathbb Z)$-equivalent to $f$,
we have $\Delta(E_{f,P})=16\Disc(f)$.
The corresponding discriminant interval for the models is
therefore $16c_-X\leq|\Delta(E)|<16c_+X$.

For a fixed model $E$, every preimage under
$f\mapsto E_{f,P}$ belongs to the
$\GL_2(\mathbb Z)$-orbit of $F_E$.
This orbit occurs in $\mathcal Fv^{(i)}$ with total
multiplicity at most $n_i$. Consequently,
\[
\frac{1}{n_i}N_P^{(i)}(X)
\leq C_P^{(i)}(X)\leq N_P^{(i)}(X),
\qquad
C_{P,T}^{(i)}(X)\leq N_{P,T}^{(i)}(X).
\]
Moreover, every model outside $\mathscr E_{P,T}^{(i)}(X)$
has a preimage outside $\mathscr S_{P,T}^{(i)}(X)$, so
\[
C_P^{(i)}(X)-C_{P,T}^{(i)}(X)
\leq N_P^{(i)}(X)-N_{P,T}^{(i)}(X).
\]

\begin{theorem}[Escape and non-escape of mass for Weierstrass models]
\label{cor:mass-transfer}
Let $i\in\{0,1\}$.
For $P\in\{(1,0),(-1,0)\}$, the model counts defined above
satisfy
\[
\lim_{T\to\infty}\liminf_{X\to\infty}
\frac{C_{P,T}^{(i)}(X)}{C_P^{(i)}(X)}=1.
\]
For every other fixed primitive vector $P$, they satisfy
\[
\lim_{T\to\infty}\limsup_{X\to\infty}
\frac{C_{P,T}^{(i)}(X)}{C_P^{(i)}(X)}=0.
\]
\end{theorem}

\begin{proof}
The preceding inequalities give
\[
\frac{C_{P,T}^{(i)}(X)}{C_P^{(i)}(X)}
\leq n_i\frac{N_{P,T}^{(i)}(X)}{N_P^{(i)}(X)}
\]
and
\[
1-\frac{C_{P,T}^{(i)}(X)}{C_P^{(i)}(X)}
\leq n_i\left(
1-\frac{N_{P,T}^{(i)}(X)}{N_P^{(i)}(X)}
\right).
\]
The first inequality transfers non-escape from
Theorem~\ref{thm:non-escape-fixed-P}; the second transfers
escape from Theorem~\ref{thm:escape-e1}.
\end{proof}

The marked solution is essential for obtaining an exact correspondence. When the marking is forgotten, the multiplicity is controlled by the following theorem of Evertse.

\begin{theorem}[Uniform bound for Thue solutions]
\label{Evertse}
Let $f(X,Y)$ be an irreducible integral binary cubic form
with nonzero discriminant. Then $f(X,Y)=1$ has at most
twelve integral solutions.
\end{theorem}

For positive discriminant, this is Evertse's bound
\cite{Evertse1983}. For negative discriminant, the classical
bound of Delone \cite{delaunay1930darstellung} and Nagell \cite{nagell1928darstellung} is five; see also the discussion
in \cite{Evertse1983}.

\begin{remark}
When the marked solution is forgotten, a fixed irreducible binary
cubic form $f$ may occur in several pairs $(f,P)$. By
Theorem~\ref{Evertse}, the equation
\[
f(P)=1
\]
has at most twelve integral solutions. Hence the forgetful map
\[
(f,P)\longmapsto f
\]
has fibres of size at most twelve.

When forms are counted in the fundamental multiset
$\mathcal Fv^{(i)}$, each $\GL_2(\mathbb Z)$-orbit occurs
with total multiplicity at most $n_i$.
Consequently, the map $(f,P)\mapsto[f]$ that forgets
the marking and retains only the $\GL_2(\mathbb Z)$-orbit
has fibres of size at most $12n_i$, with occurrences
counted according to the multiset multiplicity.
\end{remark}

We now transfer the averaged non-escape theorem to normalized
integral Weierstrass models. Fix $0<\delta<1/12$, and let
$Y=Y(X)$ be real and satisfy
\[
2\leq Y\leq X^{1/12-\delta}.
\]
Recall that
\[
\mathcal P(Y)
=
\left\{
(\alpha,\beta)\in\mathbb Z_{\mathrm{prim}}^2:
Y\leq|\alpha|,|\beta|<2Y
\right\}.
\]

Let
\[
\mathscr M^{(i)}(X;Y)
\]
denote the multiset of pairs $(f,P)$ such that
\[
P\in\mathcal P(Y),
\qquad
f\in\mathcal Fv^{(i)}
\cap V_{\mathbb Z}^{(i),\mathrm{irr}},
\qquad
f(P)=1,
\qquad
c_-X\leq|\Disc(f)|<c_+X.
\]
Similarly, let
\[
\mathscr M_T^{(i)}(X;Y)
\]
denote the submultiset consisting of the pairs for which
\[
f\in\mathcal F_Tv^{(i)}.
\]
Thus
\[
\#\mathscr M^{(i)}(X;Y)
=
J^{(i)}(X;Y)
\]
and
\[
\#\mathscr M_T^{(i)}(X;Y)
=
J_T^{(i)}(X;Y).
\]

For every pair $(f,P)$ in $\mathscr M^{(i)}(X;Y)$, the normalization
procedure of Proposition~\ref{prop:direct-correspondence} produces a
normalized integral Weierstrass model
\[
E_{f,P}:\ y^2=x^3+ux^2+Ax+B,
\qquad
u\in\{0,1,2\}.
\]
Define
\[
\mathscr E^{(i)}(X;Y)
=
\left\{
E_{f,P}:
(f,P)\in\mathscr M^{(i)}(X;Y)
\right\},
\]
where the right-hand side is regarded as a set of distinct normalized
integral Weierstrass models. Define similarly
\[
\mathscr E_T^{(i)}(X;Y)
=
\left\{
E_{f,P}:
(f,P)\in\mathscr M_T^{(i)}(X;Y)
\right\}.
\]
Thus a model belongs to $\mathscr E_T^{(i)}(X;Y)$ if at
least one of its preimages in $\mathscr M^{(i)}(X;Y)$
belongs to the cuspidal submultiset.
Set
\[
C^{(i)}(X;Y)
:=
\#\mathscr E^{(i)}(X;Y),
\qquad
C_T^{(i)}(X;Y)
:=
\#\mathscr E_T^{(i)}(X;Y).
\]

\begin{theorem}\label{thm:non-escape-averaged-models}
Fix $0<\delta<1/12$, and let $Y=Y(X)$ be real and satisfy
\[
2\leq Y\leq X^{1/12-\delta}.
\]
Then
\[
\lim_{T\to\infty}
\limsup_{X\to\infty}
\frac{C_T^{(i)}(X;Y)}
{C^{(i)}(X;Y)}
=
0.
\]
Thus the distinct normalized integral Weierstrass models arising from
solutions
\[
P\in\mathcal P(Y)
\]
exhibit averaged non-escape of mass.
\end{theorem}

\begin{proof}
The map
\[
\pi:
\mathscr M^{(i)}(X;Y)
\longrightarrow
\mathscr E^{(i)}(X;Y),
\qquad
(f,P)\longmapsto E_{f,P},
\]
is surjective by definition. We first bound the cardinality of its
fibres.

Fix a normalized integral Weierstrass model
\[
E\in\mathscr E^{(i)}(X;Y).
\]
Every pair $(f,P)\in\pi^{-1}(E)$ has $f$ in the
$\GL_2(\mathbb Z)$-orbit of the binary cubic form $F_E$ associated to
$E$. Since $\mathcal Fv^{(i)}$ is a fundamental multiset of
multiplicity at most $n_i$, this orbit occurs in
$\mathcal Fv^{(i)}$ with total multiplicity at most $n_i$.

For each such representative $f$, Theorem~\ref{Evertse} shows that
the equation
\[
f(P)=1
\]
has at most twelve integral solutions. Restricting to
$P\in\mathcal P(Y)$ can only decrease this number. Therefore,
\[
\#\pi^{-1}(E)\leq 12n_i.
\]
It follows that
\[
C^{(i)}(X;Y)
\leq
J^{(i)}(X;Y)
\leq
12n_iC^{(i)}(X;Y).
\]
In particular,
\[
C^{(i)}(X;Y)
\geq
\frac{1}{12n_i}J^{(i)}(X;Y).
\]

The restriction of $\pi$ to
$\mathscr M_T^{(i)}(X;Y)$ is also surjective onto
$\mathscr E_T^{(i)}(X;Y)$. Hence
\[
C_T^{(i)}(X;Y)
\leq
J_T^{(i)}(X;Y).
\]
Combining the preceding inequalities, we obtain
\[
\frac{C_T^{(i)}(X;Y)}
{C^{(i)}(X;Y)}
\leq
12n_i
\frac{J_T^{(i)}(X;Y)}
{J^{(i)}(X;Y)}.
\]

By Theorem~\ref{thm:non-escape-averaged},
\[
\lim_{T\to\infty}
\limsup_{X\to\infty}
\frac{J_T^{(i)}(X;Y)}
{J^{(i)}(X;Y)}
=
0.
\]
It follows that
\[
\lim_{T\to\infty}
\limsup_{X\to\infty}
\frac{C_T^{(i)}(X;Y)}
{C^{(i)}(X;Y)}
=
0.
\]
\end{proof}
\begin{remark}[Non-escape and equidistribution]
\label{rem:non-escape-without-equidistribution}
Non-escape of mass does not always imply
equidistribution. For a fixed solution
$P=(\alpha,\beta)\neq(\pm1,0)$,
Theorem~\ref{thm:non-escape-fixed-P} establishes non-escape,
while the condition $f(P)=1$ confines the forms to the
affine hyperplane
$\mathcal H_P=\{(a,b,c,d):
a\alpha^3+b\alpha^2\beta+c\alpha\beta^2+d\beta^3=1\}$.
Thus the family lies on the slice
$\mathcal H_P\cap\mathcal Fv^{(i)}$.

Suppose that the contribution from solutions in each dyadic
range $Z\leq H(P)<2Z$ is $O(X^{3/4}/Z)$ (see Theorem~\ref{thm:non-escape-averaged}), uniformly over
the ranges under consideration. Summing over $Z=2^jY$ gives
\[
\sum_{j\geq0}\frac{CX^{3/4}}{2^jY}
=
\frac{2CX^{3/4}}{Y}.
\]
As the total count is $\gg X^{3/4}$, choosing $Y$ sufficiently
large makes the proportion of forms whose integral solutions all have
height greater than $Y$ arbitrarily small. The remaining forms
lie on finitely many arithmetic slices, which have measure zero. One can choose a finite union of balls
away from the limiting slices with positive proportional
volume exceeding this exceptional proportion. Such a region
contains too few forms for equidistribution to hold.
\end{remark}

\section{Questions and further directions}
\label{sec:questions}

We conclude with three questions motivated by the numerical
experiments in Appendix~\ref{app:thin-families} and by the escape
and non-escape results proved above. Let
$E:y^2=x^3+Ax+B$ be an integral short Weierstrass model whose
associated binary cubic form
$F_E(X,Y)=X^3+AXY^2+BY^3$ is irreducible over $\mathbb Q$.
Fix a rule that assigns to each model a reduction matrix
$U_E\in\SL_2(\mathbb Z)$ such that
$f_E=U_E\cdot F_E$ is the chosen reduced representative.
This rule includes a choice whenever more than one reduction
matrix is possible. Define
$P_E=(1,0)U_E^{-1}$, so that $f_E(P_E)=1$.
All statements below concerning $P_E$ refer to this fixed rule.

For models with $|A|\asymp H^2$ and $|B|\asymp H^3$, the
experiments suggest a change in the distinguished solution near
the discriminant scale $|\Delta(E)|\asymp H^4$.
In the samples with discriminant larger than this scale,
$P_E$ was frequently one of the vectors $(1,0)$ or $(-1,0)$.
Below this scale, the sampled distinguished solutions often
had both coordinates nonzero and were still relatively small.
These observations motivate questions about both the size of
$P_E$ and the cuspidal behavior of the associated reduced forms.

The experiments use Cremona's reduction algorithm.
To compare the observed vectors with the vectors $P_E$ defined
here, one must match the action conventions and the choices
made during reduction. We therefore regard the experiments
as motivation for the questions below.

Fix positive constants
$a_-<a_+$, $b_-<b_+$, and $d_-<d_+$.
For $m\in(1,6)$ and $i\in\{0,1\}$, let
$\mathcal E_m^{(i)}(H)$ be the set of integral short
Weierstrass models $E:y^2=x^3+Ax+B$ satisfying
\[
a_-H^2\leq |A|<a_+H^2,\qquad
b_-H^3\leq |B|<b_+H^3,\qquad
d_-H^m\leq|\Delta(E)|<d_+H^m,
\]
with $(-1)^i\Delta(E)>0$ and $F_E$ irreducible over
$\mathbb Q$. These constants remain fixed as $H\to\infty$.
We count models as equations, rather than up to isomorphism.
Set
\[
C_m^{(i)}(H)=\#\mathcal E_m^{(i)}(H),
\qquad
C_{m,T}^{(i)}(H)
=
\#\{E\in\mathcal E_m^{(i)}(H):
f_E\in\mathcal F_Tv^{(i)}\}.
\]
Every limiting assertion below is restricted to the values
of $m$ and $i$ for which $C_m^{(i)}(H)\to\infty$.

The preceding mass theorems concern families defined by
conditions on the marked solution. The families considered
here satisfy additional restrictions on $A$, $B$, and the
discriminant. Escape or non-escape for a larger family need
not hold for a thin subfamily. Thus the earlier theorems
motivate, but do not settle, the following questions.

\begin{question}[Large-discriminant escape]
\label{ques:large-discriminant-escape}
Does there exist $\delta_0\in(0,3)$ such that, for every
$\delta\in(\delta_0,3)$ and every $i\in\{0,1\}$, the family
$\mathcal E_{4+\delta}^{(i)}(H)$ exhibits escape of mass?
More precisely, is
\[
\lim_{T\to\infty}\liminf_{H\to\infty}
\frac{C_{4+\delta,T}^{(i)}(H)}
     {C_{4+\delta}^{(i)}(H)}
=1?
\]
\end{question}

This asks whether, sufficiently far above the scale $H^4$,
the proportion of models whose associated reduced form
lies in any fixed truncation tends to zero.
The question is motivated by the observed occurrence of
$P_E=(\pm1,0)$ and by Theorem~\ref{thm:escape-e1}.

\begin{question}[Small-discriminant non-escape]
\label{ques:small-discriminant-non-escape}
Does there exist $\delta_0\in(0,3)$ such that, for every
$\delta\in(\delta_0,3)$ and every $i\in\{0,1\}$, the family
$\mathcal E_{4-\delta}^{(i)}(H)$ exhibits non-escape of mass?
More precisely, is
\[
\lim_{T\to\infty}\limsup_{H\to\infty}
\frac{C_{4-\delta,T}^{(i)}(H)}
     {C_{4-\delta}^{(i)}(H)}
=0?
\]
\end{question}

This asks whether, sufficiently far below the scale $H^4$,
a sufficiently large fixed truncation contains an
arbitrarily high proportion of the associated reduced forms
as $H\to\infty$.
The numerical occurrence of relatively small distinguished
solutions different from $(\pm1,0)$ motivates comparison
with Theorem~\ref{thm:non-escape-fixed-P}.

The third question concerns averaged non-escape of mass in
small-discriminant families.

\begin{question}[Averaged non-escape in small-discriminant families]
\label{conj:non-escape-small-discriminant}
There exists $\delta_0\in(0,3)$ such that the following holds.
Fix $\delta\in(\delta_0,3)$, $i\in\{0,1\}$, and
$\eta\in(0,1/12)$, and put $m=4-\delta$.
Let $Y=Y(H)$ satisfy
\[
2\leq Y\leq H^{m(1/12-\eta)},
\]
and define
\[
\mathcal P(Y)
=
\{(\alpha,\beta)\in\mathbb Z_{\mathrm{prim}}^2:
Y\leq|\alpha|,|\beta|<2Y\}.
\]
Set
\[
J_m^{(i)}(H;Y)
=
\sum_{P\in\mathcal P(Y)}
\#\{E\in\mathcal E_m^{(i)}(H):f_E(P)=1\}
\]
and
\[
J_{m,T}^{(i)}(H;Y)
=
\sum_{P\in\mathcal P(Y)}
\#\left\{
E\in\mathcal E_m^{(i)}(H):
f_E(P)=1,\quad
\operatorname{Im}(\tau_{f_E})>T
\right\},
\]
where $f_E$ is the chosen reduced form associated to $E$
and $\tau_{f_E}$ is its reduced shape point.
Then, whenever $J_m^{(i)}(H;Y)\to\infty$,
\[
\lim_{T\to\infty}\limsup_{H\to\infty}
\frac{J_{m,T}^{(i)}(H;Y)}{J_m^{(i)}(H;Y)}
=0.
\]
\end{question}

This question asks whether averaged non-escape persists
after restricting to the small-discriminant families
$\mathcal E_m^{(i)}(H)$. The averaging is over all primitive
solutions of $f_E(P)=1$ in $\mathcal P(Y)$, rather than
only the distinguished solution $P_E$.
Each model is counted with the number of such solutions.

The range of $Y$ parallels
Theorem~\ref{thm:non-escape-averaged}: since
$|\Disc(f_E)|=|\Delta(E)|/16\asymp H^m$, the discriminant
scale $X$ is replaced by $H^m$.
The theorem motivates this question but does not establish
non-escape under the additional restrictions defining
$\mathcal E_m^{(i)}(H)$.

\clearpage 
\begin{appendices}
\section{Thin families of elliptic curves}\label{app:thin-families}

In this appendix, we investigate families of elliptic curves presented by integral short Weierstrass models for which the discriminant has atypical size relative to the height. We first construct thin families in which cancellation between the terms $4A^3$ and $27B^2$ forces the discriminant to be substantially smaller than its generic scale. We then associate a binary cubic form to each model and review Cremona's reduction algorithm, which is defined by reducing an appropriate quadratic covariant in the upper half-plane.

We examine how the discriminant scale affects the roots of
$x^3+Ax+B$, the coefficients of the Cremona-reduced form, and
the associated reduced covariant point in the upper half-plane.
These numerical experiments motivated both our study of escape
and non-escape of mass in Section~\ref{sec:mass-thue} and the questions posed in
Section~\ref{sec:questions}. All experiments reported here were carried out using
SageMath.

Let $E$ be an elliptic curve over $\mathbb Q$ presented by a
short Weierstrass equation
\[
E:\ y^2=x^3+Ax+B,
\qquad A,B\in\mathbb Z.
\]
We define the naive height of this model by
$H(E)=\max\{|A|^3,|B|^2\}$. Thus
$|A|\leq H(E)^{1/3}$ and $|B|\leq H(E)^{1/2}$.
The discriminant of $E$ is
\[
\Delta(E)=-16\bigl(4A^3+27B^2\bigr).
\]

For the families considered below, we introduce a scale
parameter $H\geq1$ by requiring $|A|\asymp H^2$ and
$|B|\asymp H^3$. With this normalization, $H(E)\asymp H^6$.
If the two terms $4A^3$ and $27B^2$ do not nearly cancel,
one expects
\[
|\Delta(E)|\asymp |A|^3+|B|^2\asymp H^6.
\]

\begin{table}[H]
\centering
\renewcommand{\arraystretch}{1.3}
\begin{tabular}{|c||c||c||c|}
\hline
$A$ & $B$ & $H$ & $\Delta(E)$ \\
\hline
$-365$ & $2674$ & $10$ & $23216768$ \\
\hline
$-34877$ & $2496714$ & $10^2$ & $22265043136640$ \\
\hline
$-3218335$ & $2222357657$ & $10^3$ & $-11647533605319023$ \\
\hline
\end{tabular}
\caption{Examples of elliptic curves without cancellation in the discriminant.}
\label{tab:no-cancellation}
\end{table}

Our goal is to construct thin families for which the discriminant is substantially smaller than the generic scale $H^6$.

\subsection{Construction of elliptic curves with prescribed discriminant size}

Fix $H\geq1$, $m\in(1,6)$, and constants $0<c<c'$.
We seek integral short Weierstrass models satisfying
$cH^m\leq|\Delta(E)|\leq c'H^m$.

Choose
\[
B\in[H^3,3H^3]\cap\mathbb Z,
\]
so that $|B|\asymp H^3$. We impose the condition
\[
4A^3+27B^2
\in
\left[
\frac{cH^m}{16},
\frac{c'H^m}{16}
\right]
\]
or
\[
4A^3+27B^2
\in
\left[
-\frac{c'H^m}{16},
-\frac{cH^m}{16}
\right].
\]
The first interval gives negative elliptic discriminant,
and the second gives positive elliptic discriminant.

Equivalently, $A^3$ must lie in one of the intervals
\[
\left[
\frac{cH^m}{64}-\frac{27}{4}B^2,
\frac{c'H^m}{64}-\frac{27}{4}B^2
\right]
\]
or
\[
\left[
-\frac{c'H^m}{64}-\frac{27}{4}B^2,
-\frac{cH^m}{64}-\frac{27}{4}B^2
\right].
\]
Since $B\asymp H^3$, these intervals are centred at a point of size
\[
-\frac{27}{4}B^2\asymp-H^6.
\]
Consequently, every admissible integer $A$ satisfies
\[
|A|\asymp H^2.
\]

Whenever one of the above intervals contains an integral cube $A^3$, the corresponding model satisfies
\[
|\Delta(E)|
=
16|4A^3+27B^2|
\asymp
H^m.
\]

The admissible interval for $A^3$ has length comparable to
$H^m$. Near the relevant negative values of $A$, where
$|A|\asymp H^2$, consecutive cubes have spacing
$(A+1)^3-A^3\asymp H^4$.
Equivalently, the corresponding interval for $A$ has
length comparable to $H^{m-4}$.
For fixed $m>4$, this length grows with $H$.
For fixed $m<4$, it tends to zero, so the interval contains
at most one integer for sufficiently large $H$, and may
contain none. This gives an arithmetic reason for the
threshold $m=4$.

\begin{table}[H]
\centering
\renewcommand{\arraystretch}{1.3}
\begin{tabular}{|c||c||c||c||c|}
\hline
$A$ & $B$ & $m$ & $H$ & $\Delta(E)$ \\
\hline
$-288$ & $1883$ & $4.5$ & $10$ & $-2913840$ \\
\hline
$-209$ & $1163$ & $4$ & $10$ & $-32752$ \\
\hline
$-29315$ & $1931892$ & $4.5$ & $10^2$ & $870921152$ \\
\hline
$-22716$ & $1317789$ & $4$ & $10^2$ & $88331472$ \\
\hline
$-36211$ & $2652216$ & $3.5$ & $10^2$ & $7360192$ \\
\hline
$-2131999$ & $1198198016$ & $4.5$ & $10^3$ & $35474984081344$ \\
\hline
$-2740356$ & $1746057703$ & $4$ & $10^3$ & $1198527590736$ \\
\hline
$-3658962$ & $2693921661$ & $3.5$ & $10^3$ & $30911386320$ \\
\hline
$-3803630$ & $2855259004$ & $3$ & $10^3$ & $730297088$ \\
\hline
\end{tabular}
\caption{Examples of elliptic curves with cancellation in the discriminant.}
\label{tab:cancellation}
\end{table}

Comparing Tables~\ref{tab:no-cancellation} and~\ref{tab:cancellation}, we see that cancellation produces models whose discriminants are substantially smaller than those of generic models of comparable height. This leads naturally to the question of how such thinness affects the associated binary cubic forms and their reduction.

\subsection{The modular group and its fundamental domain}

Let $\mathfrak H=\{z\in\mathbb C:\Im(z)>0\}$ be the upper
half-plane. The group $\SL_2(\mathbb R)$ acts on $\mathfrak H$
by Möbius transformations:
\[
g\cdot z=\frac{az+b}{cz+d},
\qquad
g=\begin{pmatrix}a&b\\c&d\end{pmatrix}\in\SL_2(\mathbb R).
\]
A direct calculation gives
$\Im(g\cdot z)=\Im(z)/|cz+d|^2$, so the upper half-plane
is preserved.

Since $-I$ acts trivially, we consider the modular group
$G=\SL_2(\mathbb Z)/\{\pm I\}$. The matrices
\[
S=\begin{pmatrix}0&-1\\1&0\end{pmatrix},
\qquad
T=\begin{pmatrix}1&1\\0&1\end{pmatrix}
\]
act by $S\cdot z=-1/z$ and $T\cdot z=z+1$.
We consider the standard modular fundamental domain
\[
\mathfrak F
=\left\{z\in\mathfrak H:|z|\geq1,\ |\Re(z)|\leq\tfrac12\right\}.
\]
The following theorems and corollary can be found in
Chapter~VII of \cite{serre2012course}.

\begin{theorem}\label{thm:modular-fundamental-domain}
The set $\mathfrak F$ is a fundamental domain for the action of $G$ on $\mathfrak H$. More precisely:
\begin{enumerate}
\item every $G$-orbit meets $\mathfrak F$;
\item two distinct points of $\mathfrak F$ can be equivalent only through the usual boundary identifications.
\end{enumerate}
\end{theorem}

\begin{corollary}
The canonical map
\[
\mathfrak F
\longrightarrow
G\backslash\mathfrak H
\]
is surjective, and its restriction to the interior of $\mathfrak F$ is injective.
\end{corollary}

\begin{theorem}
The modular group $G$ is generated by $S$ and $T$.
\end{theorem}

\subsection{Reduction in the upper half-plane}

Let $z_0\in\mathfrak H$. We describe the standard algorithm
for moving $z_0$ into $\mathfrak F$. Set $z=z_0$ and $U=I_2$.

First choose an integer $k$ nearest to $\Re(z)$ and replace
\[
z\longleftarrow z-k,
\qquad
U\longleftarrow T^{-k}U.
\]
Then $|\Re(z)|\leq\frac12$. If $|z|<1$, replace
\[
z\longleftarrow-\frac1z,
\qquad
U\longleftarrow SU.
\]
We repeat these two steps until $z\in\mathfrak F$.
At the end of the procedure, $z_F=U\cdot z_0\in\mathfrak F$.

\begin{lemma}
The reduction algorithm terminates after finitely many steps.
\end{lemma}

\begin{proof}
Translations leave the imaginary part unchanged. An inversion
is applied only when $|z|<1$, in which case
\[
\Im\left(-\frac1z\right)
=\frac{\Im(z)}{|z|^2}>\Im(z).
\]
Thus every inversion strictly increases the imaginary part.

Suppose that infinitely many inversions occur. Every point
produced by the algorithm lies in the $\SL_2(\mathbb Z)$-orbit
of $z_0$, so it has the form $\gamma\cdot z_0$, where
\[
\gamma=\begin{pmatrix}a&b\\c&d\end{pmatrix}\in\SL_2(\mathbb Z),
\qquad
\Im(\gamma\cdot z_0)=\frac{\Im(z_0)}{|cz_0+d|^2}.
\]
The imaginary parts after successive inversions form a strictly
increasing sequence and are bounded below by a fixed positive
number. Consequently, the quantities $|cz_0+d|$ are bounded
above. Because $\Im(z_0)>0$, only finitely many integer pairs
$(c,d)$ satisfy such a bound. This contradicts the existence
of infinitely many strict increases. Hence the algorithm
terminates.
\end{proof}

\subsection{Covariants and Cremona reduction}

Let $E:\ y^2=x^3+Ax+B$ be an elliptic curve with integer
coefficients. We associate to $E$ the binary cubic form
$F_E(X,Y)=X^3+AXY^2+BY^3$. More generally, let
$g(X,Y)=aX^3+bX^2Y+cXY^2+dY^3$ be an integral binary cubic
form. Its discriminant is
\[
\Disc(g)=b^2c^2-4ac^3-4b^3d-27a^2d^2+18abcd.
\]
For the form associated to the elliptic curve, one has
$\Delta(E)=16\Disc(F_E)$. We use the quadratic covariants
described in \cite{cremona1999reduction}. For an irreducible
binary cubic form $g$, the construction of the associated
point $z_g\in\mathfrak H$ depends on the sign of $\Disc(g)$.

\subsubsection*{Positive discriminant}

Suppose that $\Disc(g)>0$. The Hessian quadratic form of $g$ is
\[
H_g(X,Y)
=(b^2-3ac)X^2+(bc-9ad)XY+(c^2-3bd)Y^2.
\]
Its discriminant satisfies $\Disc(H_g)=-3\Disc(g)$.
Hence $H_g(X,1)$ has negative discriminant and therefore
a unique root $z_g\in\mathfrak H$.

\subsubsection*{Negative discriminant}

Suppose that $\Disc(g)<0$. Then $g(X,1)$ has a unique real
root, which we denote by $\alpha$. Associated to $g$ and
$\alpha$ is the quadratic covariant $J_2(X)=h_0X^2+h_1X+h_2$,
where
\[
\begin{aligned}
h_0&=9a^2\alpha^2+6ab\alpha+6ac-b^2,\\
h_1&=6ab\alpha^2+6(b^2-ac)\alpha+2bc,\\
h_2&=3ac\alpha^2+3(bc-3ad)\alpha+2c^2-3bd.
\end{aligned}
\]
The discriminant of $J_2$ satisfies
$\Disc(J_2)=-4\Disc(H_g)=12\Disc(g)$.
Since $\Disc(g)<0$, we have $\Disc(J_2)<0$.
Thus $J_2$ has a unique root $z_g\in\mathfrak H$.

\subsubsection*{The Cremona-reduced form}

In either discriminant case, apply the upper-half-plane
reduction algorithm to $z_g$. Let $U\in\SL_2(\mathbb Z)$
be the resulting matrix, so that $U\cdot z_g\in\mathfrak F$.

With the action convention used throughout this paper,
the corresponding Cremona-reduced binary cubic form is
$f=U\cdot g$. By equivariance of the relevant quadratic
covariant, its associated upper-half-plane point is
\[
z_f=U\cdot z_g\in\mathfrak F.
\]
Since $\det(U)=1$, the discriminant is preserved:
$\Disc(f)=\Disc(g)$.
As shown in Tables~\ref{tab:positive-reduced-forms}
and~\ref{tab:negative-reduced-forms}, our numerical experiments
suggest that greater cancellation in the discriminant is
associated with smaller coefficients of the reduced binary
cubic forms.

\begin{table}[H]
\centering
\renewcommand{\arraystretch}{1.3}
\begin{tabular}{|c||c||c||c||c||c||c|}
\hline
$A$ & $B$ & $a$ & $b$ & $c$ & $d$ & $\Delta(E)$ \\
\hline
$-3218786$ & $2222357657$ & $1$ & $3108$ & $1102$ & $-369983$ & $710657395270177616$ \\
\hline
$-3267442$ & $2273316211$ & $53$ & $2525$ & $1759$ & $-712$ & $981855615255760$ \\
\hline
$-2433084$ & $1460776476$ & $73$ & $438$ & $-456$ & $-548$ & $5638030804224$ \\
\hline
$-2242607$ & $1292640842$ & $8$ & $-52$ & $-465$ & $296$ & $77565250304$ \\
\hline
\end{tabular}
\caption{Cremona-reduced binary cubic forms associated to elliptic curves of positive discriminant.}
\label{tab:positive-reduced-forms}
\end{table}

\begin{table}[H]
\centering
\renewcommand{\arraystretch}{1.3}
\begin{tabular}{|c||c||c||c||c||c||c|}
\hline
$A$ & $B$ & $a$ & $b$ & $c$ & $d$ & $\Delta(E)$ \\
\hline
$-3218335$ & $2222357657$ & $1$ & $3108$ & $1553$ & $97253$ & $-186360537685104368$ \\
\hline
$-3276732$ & $2283018953$ & $-138$ & $-243$ & $-2235$ & $2617$ & $-204313261711536$ \\
\hline
$-2550369$ & $1567662772$ & $-2$ & $-15$ & $-522$ & $23036$ & $-1033119582912$ \\
\hline
$-2000930$ & $1089421538$ & $44$ & $-28$ & $-25$ & $278$ & $-8900463808$ \\
\hline
\end{tabular}
\caption{Cremona-reduced binary cubic forms associated to elliptic curves of negative discriminant.}
\label{tab:negative-reduced-forms}
\end{table}

\subsection{Heuristics relating the discriminant to the geometry of the roots}

Assume that $|A|\asymp H^2$ and $|B|\asymp H^3$, and suppose
that $|\Delta(E)|\asymp H^m$, where $1<m<6$. The substitution
$x=Ht$ gives
\[
x^3+Ax+B
=H^3\left(t^3+\frac{A}{H^2}t+\frac{B}{H^3}\right).
\]
Thus the roots of the defining cubic polynomial are naturally
of size $H$.

\subsubsection*{Clustering of real roots}

Suppose that $\Delta(E)>0$. Then the polynomial
$p(x)=x^3+Ax+B$ has three distinct real roots $r_1<r_2<r_3$.
Its discriminant satisfies
\[
\Delta(E)=16(r_1-r_2)^2(r_1-r_3)^2(r_2-r_3)^2.
\]

In the generic case $m=6$, all three root separations are of
order $H$. In the thin families under consideration, the
numerical data indicate that one adjacent pair of roots
becomes closer while the other two separations remain of
order $H$. If, for example,
$|r_1-r_2|\asymp|r_1-r_3|\asymp H$, then
$|\Delta(E)|\asymp H^4|r_2-r_3|^2$. Consequently,
\[
|r_2-r_3|\asymp H^{-2}|\Delta(E)|^{1/2}\asymp H^{m/2-2}.
\]

Thus $m=4$ is a natural transition point. When $m<4$, the
relevant pair of roots approaches one another as $H$ increases.
When $m=4$, their separation remains of constant order, while
for $m>4$ their separation grows with $H$. In the generic
case $m=6$, the separation is again of order $H$.

\begin{table}[H]
\centering
\renewcommand{\arraystretch}{1.3}
\begin{tabular}{|c|c|c|c|c|c|}
\hline
$A$ & $B$ & $r_1$ & $r_2$ & $r_3$ & $m$ \\
\hline
$-3368342$ & $2279842127$ & $-2109.3101$ & $877.3135$ & $1231.9967$ & $6$ \\
\hline
$-2431361$ & $1459224931$ & $-1800.5040$ & $900.0239$ & $900.4801$ & $4$ \\
\hline
$-2242607$ & $1292640842$ & $-1729.2029$ & $864.5963$ & $864.6066$ & $3$ \\
\hline
\end{tabular}
\caption{Roots of $x^3+Ax+B$ for elliptic curves of positive discriminant. Values are rounded to four decimal places.}
\label{tab:positive-roots}
\end{table}

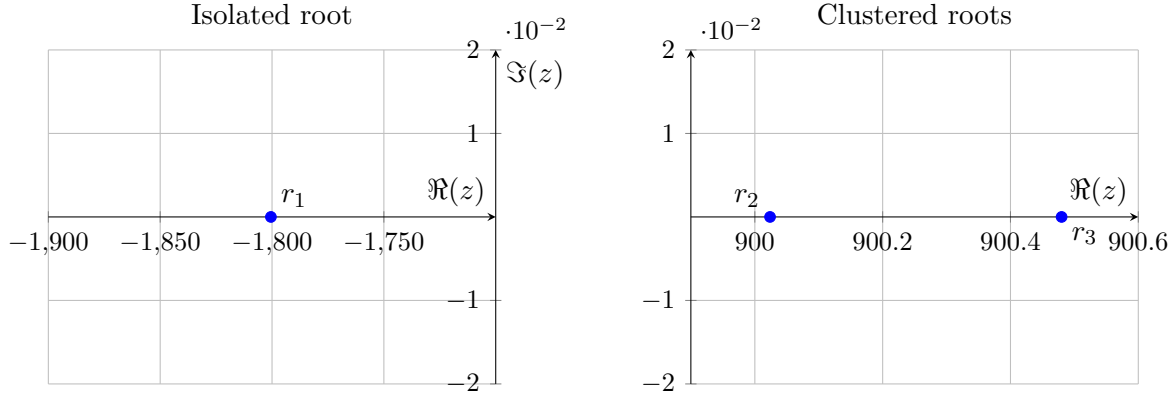
\begin{figure}[H]
\centering
\begin{tikzpicture}

\begin{axis}[
    axis lines=middle,
    xlabel={$\Re(z)$},
    ylabel={$\Im(z)$},
    xmin=-1900, xmax=-1700,
    ymin=-0.02, ymax=0.02,
    width=7.5cm,
    height=6cm,
    grid=both,
    tick label style={font=\small},
    title={Isolated root}
]

\addplot[
    only marks,
    mark=*,
    mark size=2pt,
    color=blue
] coordinates {
    (-1800.5039,0)
};

\node at (axis cs:-1800.5039,0) [above right] {$r_1$};

\end{axis}

\begin{axis}[
    at={(8.5cm,0cm)},
    anchor=south west,
    axis lines=middle,
    xlabel={$\Re(z)$},
    xmin=899.9, xmax=900.6,
    ymin=-0.02, ymax=0.02,
    width=7.5cm,
    height=6cm,
    grid=both,
    tick label style={font=\small},
    title={Clustered roots}
]

\addplot[
    only marks,
    mark=*,
    mark size=2pt,
    color=blue
] coordinates {
    (900.0239,0)
    (900.4801,0)
};

\node at (axis cs:900.0239,0) [above left] {$r_2$};
\node at (axis cs:900.4801,0) [below right] {$r_3$};

\end{axis}

\end{tikzpicture}
\caption{Real roots of $x^3-2431361x+1459224931$ in the complex plane for $m=4$. The two closely spaced roots are shown in a separate panel for clarity.}
\label{fig:roots-positive-m4}
\end{figure}

\subsubsection*{Migration of complex roots}

Suppose that $\Delta(E)<0$. Then $p(x)=x^3+Ax+B$ has one
real root and a pair of complex conjugate roots. Write
$r_1\in\mathbb R$, $r_2=u+iv$, and $r_3=u-iv$, where $v>0$.
Since $r_2-r_3=2iv$, the discriminant factorization gives
\[
|\Delta(E)|
=16|r_1-r_2|^2|r_1-r_3|^2|r_2-r_3|^2.
\]
Because $|r_1-r_2|=|r_1-r_3|$, we obtain
$|\Delta(E)|\asymp v^2|r_1-r_2|^4$.

In the families under consideration, the distance from the
real root to the complex pair remains of order $H$. Thus
$|r_1-r_2|\asymp H$, and therefore
$|\Delta(E)|\asymp v^2H^4$. It follows that
\[
v\asymp H^{-2}|\Delta(E)|^{1/2}\asymp H^{m/2-2}.
\]

Thus $m=4$ is again the transition point. When $m<4$, the
nonreal roots approach the real axis. When $m=4$, their
imaginary parts remain of constant order, while for $m>4$
their imaginary parts grow with $H$.

\begin{table}[H]
\centering
\renewcommand{\arraystretch}{1.3}
\begin{tabular}{|c|c|c|c|c|c|}
\hline
$A$ & $B$ & $r_1$ & $r_2$ & $r_3$ & $m$ \\
\hline
$-1943170$ & $1238841822$ & $-1642.3946$ & $821.1973+i\,282.7101$ & $821.1973-i\,282.7101$ & $6$ \\
\hline
$-2461913$ & $1486815829$ & $-1811.7811$ & $905.8905+i\,0.1126$ & $905.8905-i\,0.1126$ & $4$ \\
\hline
$-1922951$ & $1026361707$ & $-1601.2291$ & $800.6146+i\,0.0067$ & $800.6146-i\,0.0067$ & $3$ \\
\hline
\end{tabular}
\caption{Roots of $x^3+Ax+B$ for elliptic curves of negative discriminant. Values are rounded to four decimal places.}
\label{tab:negative-roots}
\end{table}

\begin{figure}[H]
\centering
\begin{tikzpicture}
\begin{axis}[
    axis lines=middle,
    xlabel={$\Re(z)$},
    ylabel={$\Im(z)$},
    xmin=-2000, xmax=2000,
    ymin=-0.02, ymax=0.02,
    grid=both,
    width=10cm,
    height=7cm,
    tick label style={font=\small}
]

\addplot[
    only marks,
    mark=*,
    mark size=2pt,
    color=blue
] coordinates {
    (-1601.2291,0)
};

\node at (axis cs:-1601.2291,0) [above] {$r_1$};

\addplot[
    only marks,
    mark=*,
    mark size=2pt,
    color=blue
] coordinates {
    (800.6145,0.0067)
    (800.6145,-0.0067)
};

\node at (axis cs:800.6145,0.0067) [above right] {$r_2$};
\node at (axis cs:800.6145,-0.0067) [below right] {$r_3$};

\end{axis}
\end{tikzpicture}
\caption{Roots of $x^3-1922951x+1026361707$ in the complex plane for $m=3$.}
\label{fig:roots-negative-m3}
\end{figure}
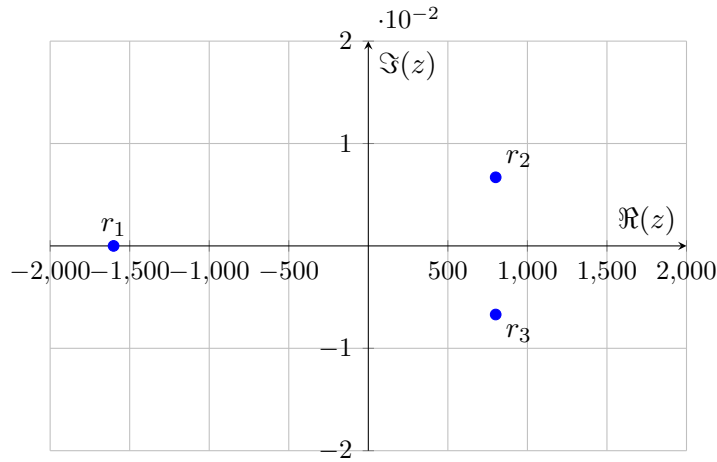

\subsection{Summary of the computational evidence}

The preceding estimates concern the roots of the cubic polynomial
\[
x^3+Ax+B.
\]
They identify $m=4$ as a natural transition point. When $m<4$, two real roots approach one another in the positive-discriminant case, while the complex conjugate roots approach the real axis in the negative-discriminant case. When $m>4$, the corresponding separations grow with $H$.

We determine the behavior of the upper-half-plane points associated to the Hessian or Julia covariants after reduction to the standard modular fundamental domain. The computations indicate that the relation between the discriminant and the height has a significant effect on the location of the reduced covariant points.

The figures below display the reduced upper-half-plane roots of the Hessian covariant for elliptic curves of positive discriminant. Figure~\ref{fig:hessian-roots-positive-overall} shows the distribution over an extended vertical range, while Figure~\ref{fig:hessian-roots-positive-zoom} gives a magnified view of the lower portion of the fundamental domain.

\begin{figure}[!htbp]
\centering
\includegraphics[width=0.6\textwidth]{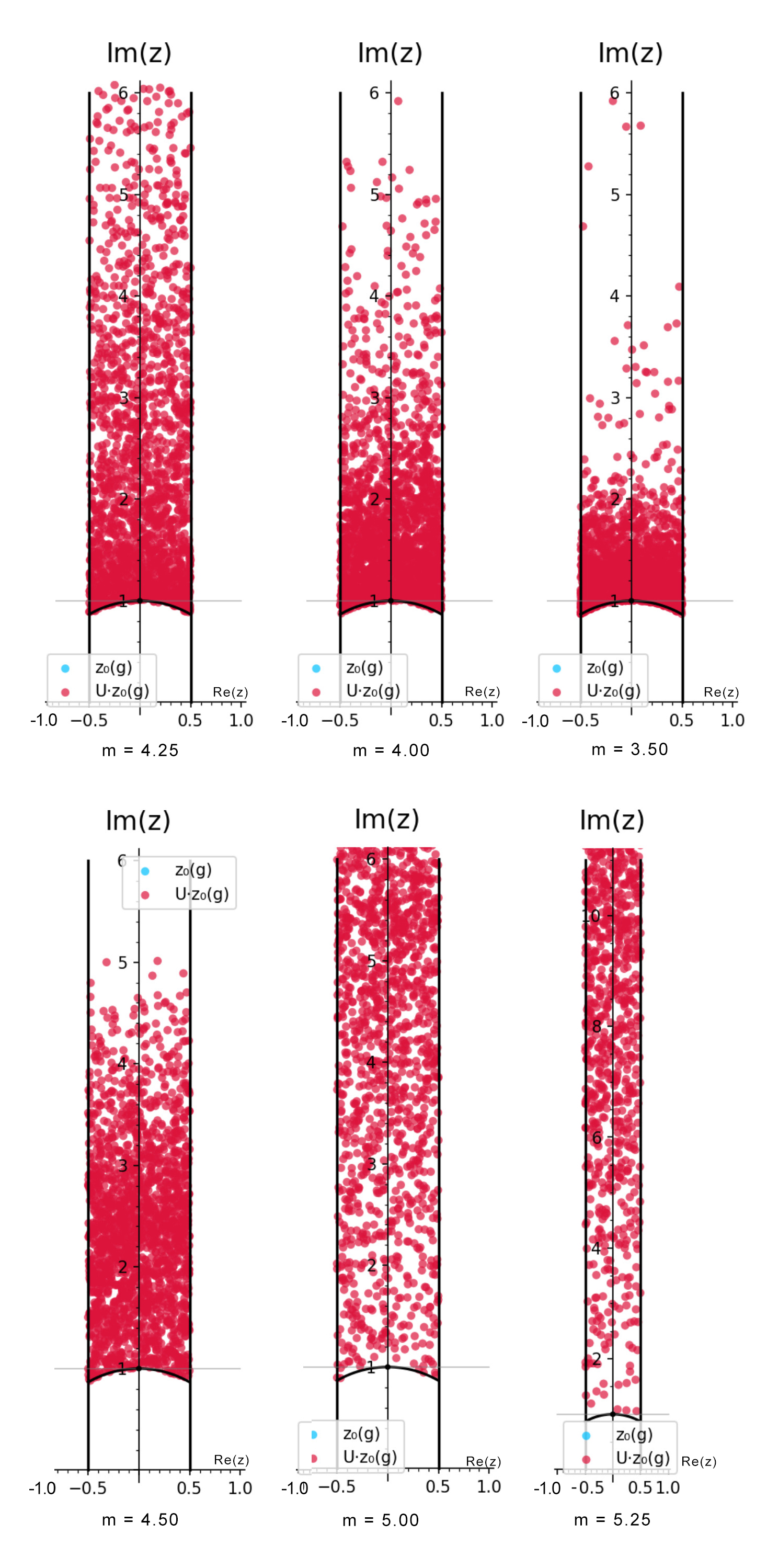}
\caption{Distribution of the reduced upper-half-plane roots of the Hessian covariant for elliptic curves of positive discriminant, shown for several values of $m$. The figure displays the points over an extended vertical range in the standard fundamental domain.}
\label{fig:hessian-roots-positive-overall}
\end{figure}

\begin{figure}[!htbp]
\centering
\includegraphics[width=0.8\textwidth]{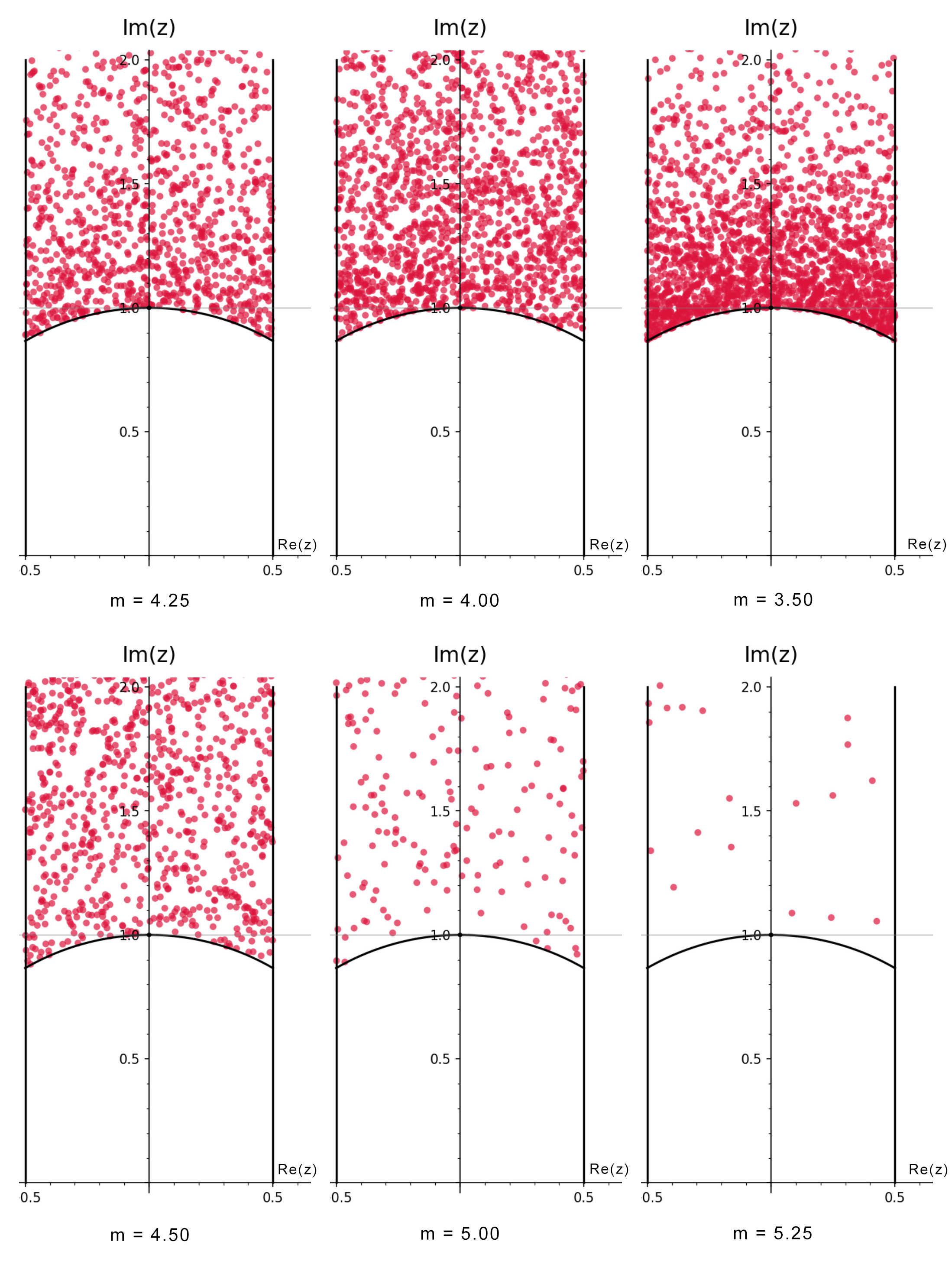}
\caption{Magnified view of the distribution of the reduced Hessian roots near the lower part of the standard fundamental domain, shown for several values of $m$.}
\label{fig:hessian-roots-positive-zoom}
\end{figure}

As illustrated in Figures~\ref{fig:hessian-roots-positive-overall} and~\ref{fig:hessian-roots-positive-zoom}, the numerical data suggest that, in the sampled small-discriminant regimes, many of the reduced points remain in the lower portion of the fundamental domain rather than moving arbitrarily far into the cusp. This motivates the study of non-escape of mass. In other parameter ranges, the associated points may move higher into the cusp, suggesting the possibility of escape of mass. These computations are exploratory and do not by themselves establish either phenomenon.

Motivated by this evidence, we seek suitable notions of escape and non-escape of mass for families arising from integral Weierstrass models. It is more convenient to formulate the counting problem in the
space of binary cubic forms, where we can use established
counting techniques, such as Bhargava's averaging method \cite{bhargava2010density}. The form
$F_E(X,Y)
=
X^3+AXY^2+BY^3$
is monic and therefore satisfies
$F_E(1,0)=1$.
After reduction, this distinguished solution is transported to another primitive integral vector. Thus the natural objects are binary cubic forms equipped with a distinguished integral solution to the corresponding Thue equation. Table~\ref{tab:negative-elliptic-thue-solutions} illustrates this behavior for elliptic curves of negative discriminant.

\begin{table}[H]
\centering
\renewcommand{\arraystretch}{1.3}
\begin{tabular}{|c|c|c|c|}
\hline
$A$ & $B$ & $P=(\alpha,\beta)$ & $m$ \\
\hline
$-3218335$ & $2222357657$ & $(1,0)$ & $5.8$  \\
\hline
$-3081085$ & $2081630455$ & $(5,-7)$ & $4$ \\
\hline
$-2042557$ & $1123593947$ & $(-15,7)$ & $3.4$ \\
\hline
$-3225792$ & $2229988659$ & $(-20,1)$ & $3$ \\
\hline
\end{tabular}
\caption{Integral short Weierstrass models of negative
discriminant and distinguished solutions
$P_E=(\alpha,\beta)$ satisfying $f_E(P_E)=1$, where
$f_E$ is the associated Cremona-reduced binary cubic form.}
\label{tab:negative-elliptic-thue-solutions}
\end{table}
These examples motivate the study of binary cubic forms
equipped with distinguished integral solutions to their
Thue equations. The counting arguments in
Section~\ref{sec:mass-thue} use the fundamental multiset
$\mathcal Fv^{(i)}$ associated with the reduction setup
of Section~\ref{sec:df-reduction}.

\end{appendices}

\clearpage

\bibliography{mybib}
\bibliographystyle{unsrt}
\end{document}